\documentclass[ijoc,sglanonrev]{informs5}
\RequirePackage{tgtermes}
\RequirePackage{newtxtext}
\RequirePackage{newtxmath}
\RequirePackage{bm}
\RequirePackage{endnotes}
\OneAndAHalfSpacedXII % Current default line spacing
\usepackage[ruled,linesnumbered]{algorithm2e}
\usepackage{natbib}
 \bibpunct[, ]{(}{)}{,}{a}{}{,}%
 \def\bibfont{\small}%
 \def\BIBand{and}%
\graphicspath{{/images/}}

\usepackage{booktabs}
\usepackage{hyperref}
\usepackage[para]{threeparttable}
\usepackage{mathtools}
\usepackage{bbm}
\usepackage{enumitem}
\usepackage{multirow}
\usepackage[caption=false]{subfig}

\EquationsNumberedThrough    % Default: (1), (2), ...
\TheoremsNumberedThrough     % Preferred (Theorem 1, Lemma 1, Theorem 2)
\ECRepeatTheorems  %  
\MANUSCRIPTNO{IJOC-0001-2026.00}

\DeclarePairedDelimiterX\Set[2]{\lbrace}{\rbrace}{ #1 \,\delimsize| \,\mathopen{} #2 }
\newcommand{\Bs}[1]{\mathbb{#1}} %for upper case mathbb
\newcommand{\Cs}[1]{\mathcal{#1}} %for upper case mathcal
\newcommand{\one}{\mathbf{1}}

\DeclareMathOperator{\dom}{dom}
\DeclareMathOperator{\dist}{dist}
\DeclareMathOperator{\Graph}{Graph}
\newcommand{\e}[1]{\Bs{E} \left[ #1 \right]}
\newcommand{\ep}[2]{\Bs{E}_{#1} \left[ #2 \right]}

\begin{document}
%\setcounter{page}{1}
%\DOI{10.1287/xxxx.0000.0000}%
\RUNAUTHOR{Abtahi, Rahimian, and Khademi}
% \RUNTITLE{Predictive Maintenance in Manufacturing}
\RUNTITLE{Facility Location with Endogenous Demand Learning}
% \TITLE{Optimal Resource Allocation in Humanitarian Logistics: A Stochastic Programming Approach}
\TITLE{A Two-Stage Stochastic Facility Location Problem with Endogenous Demand Learning}
\ARTICLEAUTHORS{%
%\AUTHOR{John Doe,\textsuperscript{a} Jane Smith,\textsuperscript{b}}
%\AFF{\textsuperscript{a}Department of Industrial Engineering, University of XYZ, \EMAIL{john.doe@xyz.edu; \textsuperscript{b}Department of Computer Science, University of ABC, \EMAIL{jane.smith@abc.edu}} 
\AUTHOR{Mahbod Abtahi}
\AFF{Department of Industrial Engineering, Clemson University, Clemson SC 29634, USA  \EMAIL{sabtahi@clemson.edu}}
\AUTHOR{Hamed Rahimian}
\AFF{Department of Industrial Engineering, Clemson University, Clemson SC 29634, USA  \EMAIL{hrahimi@clemson.edu}}
\AUTHOR{Amin Khademi}
\AFF{Department of Industrial Engineering, Clemson University, Clemson SC 29634, USA  \EMAIL{khademi@clemson.edu}}
} % end of the block
\ABSTRACT{%
In facility location problems, planners can reduce demand uncertainty by collecting data, but these actions are costly. In this paper, we ask how a planner should jointly decide where to open facilities and invest in demand learning, which comes with a cost. The challenge is that learning changes how future random costs should be 
evaluated, creating a model that is both endogenous and nonlinear in location decisions. 
We model this problem within a two-stage stochastic framework, where both location and learning decisions have to be made here-and-now. 
Under some mild assumptions, we obtain a closed-form reformulation, which further enables a theoretical analysis of the sensitivity of learning decisions to model parameters and design of convergent solution algorithms that alternate between updating location and learning decisions.
Numerical experiments on benchmark instances show that learning is most valuable when initial demand uncertainty is high, learning is effective, and learning costs are low. Moreover, the proposed solution algorithms find high-quality solutions much faster than exhaustive search and a standard commercial nonlinear solver.
These findings show how information-acquisition decisions can be built directly into strategic facility planning. They help decision makers target limited learning resources where better demand information has the greatest value.
}%

%\FUNDING{The authors gratefully acknowledges the support of the U.S. Air Force Office of Scientific Research through grant FA9550-24-1-0241.}
%Supplemental Material:
%Data Ethics & Reproducibility Note:
\KEYWORDS{Facility location, Two-stage stochastic programming, Demand learning, Nonconvex optimization, Proximal coordinate descent} 
%\HISTORY{Received: Month DD, YYYY; Accepted: Month DD, YYYY; Published Online: Month DD, YYYY}
\maketitle

%%%%%%%%%%%%%%%%%%%%%%%%%%%%%%%%%%%%%%%%%%%%%%%%%
\section{Introduction}
\label{sec: intro}

The facility location problem (FLP) is a fundamental model in operations research for deciding where to open facilities and how to allocate customer demand to them. These decisions arise in logistics and supply chain management, energy and transportation networks, emergency response, and other strategic planning problems \citep{daskin2005facility,melo2009facility}. Because facility location decisions are often costly, long-term, and difficult to reverse, the consequences of poor location choices can persist long after the original plan.
A central difficulty in FLP is that demand is rarely known with certainty at the time location decisions are made, motivating notable work on facility location with demand uncertainty. 
Two-stage stochastic FLP models address this difficulty by selecting facilities before demand is realized and then allocating demand, or taking recourse actions, after uncertainty is revealed \citep{snyder2006facility}. 
Such models usually treat the demand distribution as an exogenous input. In many applications, however, the planner can influence the quality of the demand information available before operations begin. For example, firms may collect local market data, run pilot studies, deploy sensors, purchase forecasts, or conduct targeted surveys before committing to the final location plan. These activities do not eliminate demand uncertainty, but reduce it.

The goal of this work is to propose a model for a facility location problem where the decision maker (DM) can learn about the demand with some effort, and then provide a solution methodology to the model, as well as insight about decisions and optimal costs in the presence of demand learning.
To that end, we consider a setting where the DM is able to purchase signals about the demand with the property that higher quality signals are more costly.
Our modeling approach is inspired by the framework proposed in \citet{artstein1993sensors}, in which a DM can seek more information about an optimization problem through costly ``sensors.''
For example, with an effort of $n$ with a linear cost of $d \times n$ the DM will experience a demand with a variance of $\sigma^2/n$.
Specifically, in our model, in the first stage we introduce a decision for how much costly effort the DM should invest in learning the demand variance in addition to the standard location decisions.
For the second stage, then, the demand distribution depends on how much effort is exerted in the first stage.

\subsection{Contributions}
In this paper, we study the trade-off between sampling effort and higher quality signals in a two-stage stochastic capacitated FLP, which makes the demand distribution endogenous. 
In the first stage, the DM chooses which facilities to open and allocates the sampling effort to each customer. Learning is modeled as a variance-reducing mechanism: demand means remain fixed, while customer-specific sampling efforts reduce the variances of normally distributed demands. In the second stage, demand is served subject to capacity restrictions, and unmet demand incurs a penalty, yielding a recourse structure that is separable across customers.

The main challenge is that learning makes the objective both decision-dependent and nonlinear. Opening facilities changes the recourse function, while sampling decisions change the distribution under which the recourse function is evaluated. A direct treatment of this problem leads to nonconvex mixed-integer programs (MIPs). 
Nonetheless, by exploiting the separability of the recourse function we first derive a closed-form reformulation, which facilitates the analysis of the problem. 
Specifically, under independent normal demands, this reformulation turns the learning problem into a structured nonconvex MIP with separability in the continuous learning decisions, which is amenable to theoretical and numerical analysis (Section~\ref{sec: FL_Learning}). 

Second, we conduct comparative statics analysis for the optimal learning effort and the associated optimal value with respect to the sampling cost and the demand uncertainty (Theorems~\ref{thm:cs_min} and \ref{thm:cs_value}). 
This analysis provides insight about the trade-off between sampling effort and signal quality.
Third, we develop solution algorithms that alternate between updating facility location and learning decisions, in the spirit of Gauss-Seidel methods, also known as {\it coordinate descent} methods (Section~\ref{sec: algorithms}). We use sample average and piecewise-linear approximations to handle the nonlinear terms efficiently (Section~\ref{sec: alg_SAA} and \ref{sec: alg_PWL}).
We also establish theoretical convergence of the proposed solution algorithms and consistency of the set of stationary points and local minimizers (Theorems~\ref{thm:min_convergence}--\ref{thm:pwl_stationary_consistency}).

Finally, we apply our model and analysis to standard benchmarks for FLPs by adapting them to our setting.
The numerical study then examines the value of learning, the comparative statics, and the computational performance of the proposed algorithms. The experiments show how reducing demand uncertainty through learning may be beneficial, and how partial learning may change the benefit. Moreover, it shows that the proposed solution approaches achieve strong performance in both solution quality and computational efficiency compared with enumeration and a commercial nonconvex solver. 

\subsection{Organization}
The remainder of the paper is organized as follows. Section~\ref{sec: lit} provides a review of related work. Section~\ref{sec: problem_statement} presents the problem statement, the learning model, and the closed-form reformulation. Section~\ref{sec: comparativestatics} analyzes how key model parameters affect
the optimal learning decisions and the corresponding optimal objective values. Section~\ref{sec: algorithms} introduces the proposed solution approaches. Section~\ref{sec: results} reports the numerical results, including the value of learning and computational comparisons.
Section~\ref{sec: conclusion} ends with conclusions and directions of future research.
Omitted proofs and all technical definitions from variational analysis are relegated to the Electronic Companion (EC). %can be found in Section \ref{sec: EC_prelim}. 

\noindent {\bf Notation}. 
%For $n \in \Bs{N}$, we refer to the ordered index set $\{1,\dotsc,n\}$ by $[n]$. 
%For any $n \in \Bs{N}$, $\one$ refers a vector of ones of size $n$. 
For a closed set $\Cs{S} \subseteq \Bs{R}^n$, its indicator function is defined as $i_{\Cs{S}}(x)=0$ if $x \in \Cs{S}$ and $+\infty$, otherwise.
For a function $f$, $\partial f (x)$ denotes the subdifferential of $f$.  
%For a function $f(x,y)$, we use $\frac{\partial f}{\partial x}$ and $\frac{\partial^2 f}{\partial x^2}$, and $\frac{\partial^2 f}{\partial x \partial y}$ to denote partial derivatives of $f$. 

\section{Related Work}
\label{sec: lit}

There are two main streams of literature that are related to our study: (i) the literature for learning the demand in facility location, and (ii) multistage stochastic programs with endogenous uncertainty.

\subsection{Demand Learning in Facility Location}

In this study, learning in facility location refers to an information-acquisition decision that improves the distributional description of demand before recourse decisions are made. 
This view is related to the general framework of \citet{artstein1993sensors}, where an information structure, coined as sensors, refines the distribution of uncertain parameters before optimization decisions are evaluated. In facility location, customer demand is a natural target for such learning because both location and allocation decisions depend on the  demand.

Several studies on FLP include mechanisms through which demand information becomes available. 
For example, \citet{bhatti2015alternative} study a two-stage stochastic FLP in which first-stage deployment decisions affect the time required for demand uncertainty to be revealed. 
In that setting, learning is represented by a decision-dependent timing of uncertainty resolution rather than through a decision that directly changes demand distribution or improves its precision.
\citet{ma2024probing} consider a FLP with probing decisions in which selected components of an auxiliary random vector are observed and then used to update the conditional distribution of demand. 
Their model obtains demand information indirectly through correlated observations and leads to a three-stage structure in which probing, facility decisions, and recourse decisions are made sequentially. \citet{cao2026online} study an online FLP for mobile stores, in which sales observations are used to learn demand over time and update locations sequentially, emphasizing exploration and exploitation in repeated operations.

Our model takes a different perspective. Learning is neither a change in the timing of revelation nor an online update from repeated operations. Instead, it is a costly first-stage planning decision that directly reduces customer-level demand variance before demand is realized. 
Our approach is based on the premise that sampling effort reduces the variance of distributions.
For example, in \citet{moscarini2001optimal} learning occurs by reducing the variance of the normally distributed signals in the following way: by exerting an effort of $n$ with a convex cost of $c(n)$ the DM receives normally distributed signals with a variance of $\sigma^2/n$. 
Long-lasting location and learning decisions are chosen simultaneously, and the value of learning is measured through its effect on the expected recourse cost of the stochastic FLP. This perspective allows us to study how costly learning effort should be allocated across customers.

\subsection{Decision Dependency in Stochastic Programs}

The proposed model is also related to stochastic programs (SPs) with decision-dependent uncertainty, or {\it endogenous} uncertainty. 
Classical two-stage SPs assume that the probability distribution of uncertain parameters is known in advance and does not depend on the first-stage decisions.
In contrast, decision-dependent uncertainty allows decisions to affect either the timing of uncertainty resolution or the probability distribution itself \citep{jonsbraaten1998,Goel2006,hellemo2018,vayanos2026robust}. The latter case is especially relevant here: sampling decisions determine the variance of each demand component, so the expected recourse cost is evaluated under a distribution induced by the first-stage decision.

Decision dependence has also been studied through set-based descriptions of uncertainty. In robust optimization, uncertainty set may depend on the decision 
\citep{Nohadani2018};  
in facility location, \citet{Zeng2022} study a two-stage robust model in which the demand uncertainty set depends on the opened facilities. In distributionally robust optimization, the ambiguity set of probability distributions may likewise depend on the decision 
\citep{LuoMehrotra2020,Basciftci2021,RahimianMehrotra2023}. 
These approaches capture endogenous uncertainty through sets of possible realizations or distributions, while our setting uses a parametric stochastic model in which the variance of each demand component is controlled by a costly learning decision.

Existing studies of decision-dependent probability distributions typically consider specific structures through which decisions affect the underlying probability distribution \citep{hellemo2018}. These structures allow the decision dependence to be exploited to develop reformulation and algorithmic strategies. 
For instance, \citet{bazotte2026} represent decision-dependent random vectors as transformations of first-stage decisions and exogenous random vectors, thereby enabling decision-independent reformulations and standard SP methods such as sample average approximation. 
\citet{pantuso2025shaped} study two-stage SPs in which a finite partition of the first-stage feasible region determines the relevant probability distribution and develop 
distribution-specific L-shaped cuts, see \cite{hewitt2025production} and \cite{pantuso2025solution} for applications in facility location and production planning problems, respectively.

Our setting differs in both motivation and structure: the distributional dependence is continuous, rather than induced by a predefined ``discrete selection of distributions" and is driven by costly learning effort. This leads to a nonconvex optimization problem in the joint location and learning decisions, which necessitates tailored solution approaches.

%%%%%%%%%%%%%%%%%%%%%%%%%%%%%%%%%%%%%%%%%%%%%%%%%
\section{Problem Statement and Formulation}
\label{sec: problem_statement}
%%%%%%%%%%%%%%%%%%%%%%%%%%%%%%%%%%%%%%%%%%%%%%%%%

We introduce a two-stage stochastic capacitated FLP in which the DM can better learn demand by costly effort. Section~\ref{sec: FL_formulation} presents the problem under a known demand distribution to set the stage. Section~\ref{sec: FL_Learning} formulates the problem with the demand learning framework. We derive a closed-form expression for the objective function and establish its structure in Sections~\ref{sec: ExplicitReformulation} and \ref{sec:closed-form-structure}, respectively.

\subsection{Two-Stage Stochastic Facility Location under a Known Demand Distribution}
\label{sec: FL_formulation}

Consider a set $\Cs{I}$ of potential facility locations and a set $\Cs{J}$ of customer types/points. For each $i \in \Cs{I}$, let $u_i$ denote the installation cost, with $p$ being the number of facilities to be installed. For each $j \in \Cs{J}$, let $c_{0j}$ denote the unit penalty cost of unmet demand. Moreover, for each $i \in \Cs{I}$ and $j \in \Cs{J}$, let $c_{ij}$ denote the unit procurement cost incurred 
%and revenue incurred/obtained 
by satisfying the demand of customer $j \in \Cs{J}$ from location $i \in \Cs{I}$. We assume that $c_{0j} > c_{ij}>0$ %and $r_{j} > c_{ij}$ 
for all $i \in \Cs{I}$ and $j \in \Cs{J}$. In addition, let $D_{ij}$ denote the allotted capacity of location $i \in \Cs{I}$ for customer $j \in \Cs{J}$. 
Let $\xi \in \Bs{R}^{|\Cs{J}|}$ denote the uncertain demand, and suppose that $\xi$ has probability density function (pdf) $\varpi(\xi \mid \vartheta)$, where $\vartheta$ is a known parameter vector, e.g., a multivariate normal distribution with known mean and covariance. 
%distribution $\Bs{P}$.  
A two-stage stochastic capacitated FLP can then be formulated as
\begin{equation}
    \label{eq: FL_DI_first}
    %\begin{aligned}
         \min_{x \in \Cs{X}} \ \sum_{i \in \Cs{I}} u_i x_i + \ep{\xi \mid \vartheta}{R(x,\xi)}, \\ 
        %\textrm{s.t.} \quad & \sum_{i \in \Cs{I}} x_i =p, \\   
        %& x \in \{0,1\}^{|\Cs{I}|},
    %\end{aligned}
\end{equation}
where $\Cs{X}:=\{x: \sum_{i \in \Cs{I}} x_i  \le p, \; x \in \{0,1\}^{|\Cs{I}|}\}$. Moreover, for a fixed location decision $x \in \Cs{X}$ and demand realization $\xi \in \Bs{R}^{|\Cs{J}|}$, cost function $R(x,\xi)$ is defined as 
\begin{equation}
    \label{eq: FL_DI_second}
    \begin{aligned}
         R(x, \xi) = \min_{y,s} \ & \sum_{j \in \Cs{J}} \left\{ \sum_{i \in \Cs{I}}  c_{ij} y_{ij} + c_{0j} s_{j} \right\} \\ 
        \textrm{s.t.} \quad & \sum_{i \in \Cs{I}} y_{ij} + s_j = \xi_j, \quad & j \in \Cs{J}, \\
        & y_{ij} \leq D_{ij} x_i, \quad & i \in \Cs{I}, \; j \in \Cs{J}, \\
        & s_j \geq 0, \,\,\, y_{ij} \geq 0, \quad & i \in \Cs{I}, \;  j \in \Cs{J}.  
    \end{aligned}
\end{equation}
Here, the second-stage variables are $y_{ij}$, the procured amount from location $i \in \Cs{I}$ to customer $j \in \Cs{J}$, and $s_j$, the unmet demand for customer $j \in \Cs{J}$ \citep{Basciftci2021}.

\subsection{Two-Stage Stochastic Facility Location with Demand Learning}
\label{sec: FL_Learning}

Formulation \eqref{eq: FL_DI_first} assumes that the parameter vector $\vartheta$ is known, which may not hold in practice. In many settings, the parametric family $\varpi(\xi \mid \vartheta)$ is known while $\vartheta$ itself is unknown. 
In this paper, we assume that before the demand is realized, the DM can exert effort which may yield an update on parameter $\vartheta$. 

To formalize the learning process, we assume that demand follows an uncorrelated multivariate normal distribution. For each $j \in \Cs{J}$, let $n_j$ denote the sampling effort for customer $j$, and let $d_j$ denote the corresponding unit sampling cost. A sampling effort $n=(n_1, \ldots, n_{|\Cs{J}|})$ may update the distribution of demand as $\xi \sim \textrm{N}(\mu,\Sigma_n)$. Here, $\Sigma_n$ is a diagonal matrix, with entries $h_j(n_j; \sigma_{j})^2$, where $\sigma_j$ is a fixed parameter, $j \in \Cs{J}$, collectively referred to as the demand uncertainty. 
Throughout the paper, we make the following assumption about $h_j(n_j; \sigma_{j})$, $j \in \Cs{J}$. 
\begin{assumption}[Structural Properties of Standard Deviation]
\label{assum:properties}
For any $j \in \Cs{J}$, 
\begin{enumerate}[label=(\roman*)]
    \item $h_j(\cdot;\sigma_{j})$ is finite, positive, strictly convex, decreasing, and continuously differentiable on $[0,\infty)$ for any $\sigma_{j}$.
    \item $\frac{\partial h_j}{\partial \sigma_{j}}(n_j;\sigma_{j})>0$ for all $n_j\ge 0$.
    \item $\frac{\partial^2 h_j}{\partial n_j \partial \sigma_{j}}(n_j;\sigma_{j})<0$ for all $n_j \ge 0$.
\end{enumerate}
\end{assumption}

Assumption~\ref{assum:properties} implies that for each customer $j$, an increase in the sampling effort results in a reduction in the standard deviation of demand, $h_j(\cdot,\sigma_j)$, while mean demand remains unchanged, $j \in \Cs{J}$. Moreover, the marginal effect of increasing the sampling effort $n_j$ on the standard deviation, $h_j$, decreases as the demand uncertainty $\sigma_j$ increases. 
Assumption~\ref{assum:properties} generalizes, in a natural way, the classical setting that sampling effort $n$ will result in the variance of $\sigma^2/n$.
In other words, the classical setting satisfies this assumption.

For a fixed location $x$, sampling, in turn, may result in a reduction in the expected recourse cost, $\ep{\xi}{R(x,\xi)}$, provided that $R(\cdot,\xi)$ is not linear in $\xi$. However, sampling incurs a cost, creating a trade-off. Putting this together with the typical trade-off in \eqref{eq: FL_DI_first}, between facility location and recourse costs, yields a two-stage stochastic capacitated FLP with endogenous demand learning as
\begin{equation}
    \label{eq: FL_DI_first_learning}
    %\begin{aligned}
         \min_{x \in \Cs{X}, n\in \Cs{N}} \  \left\{F(x,n):= \sum_{i \in \Cs{I}} u_i x_i + \sum_{j \in \Cs{J}} d_j n_j + \ep{\xi}{R(x,\xi)}\right\}, 
         %\textrm{s.t.} \quad & \sum_{i \in \Cs{I}} x_i =p, \\   
        %& x \in \{0,1\}^{|\Cs{I}|},\\
        %& n \in [0,b]^{|\Cs{J}|}.
    %\end{aligned}
\end{equation}
where $x \in \Cs{X}$ and $n \in \Cs{N} := [0, b]^{|\Cs{J}|}$ have to be made here-and-now, with $\xi \sim \textrm{N}(\mu, \Sigma_n)$.  Formulation \eqref{eq: FL_DI_first_learning} is the central framework for study in this paper, for which we derive a closed-form expression for the objective function $F(x,n)$ next.

\subsection{Closed-Form Expression of the Objective Function}
\label{sec: ExplicitReformulation}

Observe that given fixed \(x\) and \(\xi\), \(R(x, \xi)\), as defined in~\eqref{eq: FL_DI_second}, is decomposable across customers. 
Let \(R_j(x, \xi_j)\) denote the cost associated with customer $j \in \Cs{J}$, with $R(x, \xi) = \sum_{j \in \Cs{J}} R_j(x, \xi_j)$. Proposition~\ref{prop: recourse_j}  presents a closed-form expression for the expected recourse function for each customer. %under a fixed sampling effort and location decision. 

\begin{proposition} 
\label{prop: recourse_j}
    Consider a fixed $x$ and $n$. For customer $j \in \Cs{J}$, let $\pi$ be a permutation of $\Cs{I} \cup \{0\}$ such that $c_{\pi(0)j} > c_{\pi(1)j} \geq c_{\pi(2)j} \geq \dots \geq c_{\pi(|\Cs{I}|)j} \geq 0$. 
    Suppose that $\xi \sim N(\mu, \Sigma_n)$. Then, %we have 
    \begin{align*}
    \ep{\xi_j}{R_j(x, \xi_j)} &= a_{\pi(0)j}(x) + c_{\pi(0)j} \mu_j + h_j(n_j;\sigma_j) \sum_{k=1}^{|\Cs{I}|} \left(c_{\pi(k-1)j} - c_{\pi(k)j}\right)\Psi\!\left(\frac{\hat{\xi}_{\pi(k)j}(x) - \mu_j}{h_j(n_j;\sigma_j)}\right),
    \end{align*}
    where 
    %\begin{equation}
    %\label{eq: intercept_j}
    %    \begin{aligned}
            $ a_{\pi(0)j}(x) = \sum_{i=1}^{|\Cs{I}|} D_{\pi(i)}(c_{\pi(i)j} - c_{\pi(0)j}) \, x_{\pi(i)}$, 
    %    \end{aligned}
    %\end{equation} 
    %\begin{equation}
    %\label{eq: break_j}
    %    \begin{aligned}
            $\hat{\xi}_{\pi(k)j}(x) = \sum_{i=k}^{|\Cs{I}|} D_{\pi(i)} x_{\pi(i)}$ for $k=0, \dots, |\Cs{I}|$, 
    %    \end{aligned}
    %\end{equation}
    and 
    \begin{equation}
    \label{eq:Psi}
        \Psi(z):= z\Phi(z) + \phi(z),   
    \end{equation}
    with $\Phi(\cdot)$ and $\phi(\cdot)$ being the cdf and pdf of a standard normal random variable, respectively.
\end{proposition}

By the fact that $R(x, \xi) = \sum_{j \in \Cs{J}} R_j(x, \xi_j)$, we have $\ep{\xi}{R(x, \xi)} = \sum_{j \in \Cs{J}} \ep{\xi_j}{R_j(x, \xi_j)}$, where $\ep{\xi_j}{R_j(x, \xi_j)}$ is characterized by Proposition~\ref{prop: recourse_j}, for each $j \in \Cs{J}$. Hence, we present an abstract form of problem \eqref{eq: FL_DI_first_learning} as the following nonconvex MIP
\begin{equation}
\label{eq: abstract_model}
    \min_{x \in \Cs{X}, n \in \Cs{N}} \Big \{ F(x, n) = e+ d^\top n + r^\top x  + f(x, n)\Big\}
\end{equation}
where 
\begin{equation}
\label{eq: abstract_model_f}
    f(x, n) := \sum_{j \in \Cs{J}} h_j(n_j;\sigma_j) \sum_{i \in \Cs{I}}  q_{ij} \Psi\left(\frac{g_{ij}(x)}{h_j(n_j;\sigma_j)}\right)
\end{equation}
with $q_{ij}>0$ and $g_{ij}(x)$ being a linear function of $x$, $i \in \Cs{I}$ and $j \in \Cs{J}$. %Moreover, $\Psi'(z) = \Phi(z)$. 
%We also have $\Psi(-\infty)=0$ and $\Psi(+\infty)=+\infty$. 
We note that by Proposition~\ref{prop: recourse_j},  the constant term in \eqref{eq: abstract_model} corresponds to $\sum_{j \in \Cs{J}} c_{\pi(0)j} \mu_j$, while the linear terms in $x$ correspond to $a_{\pi(0)j}(x)$, $j \in \Cs{J}$, and fixed installation cost, $u^\top x$. 

\subsection{Structural Properties of the Objective Function}
\label{sec:closed-form-structure}

In this section, we establish structural properties of \eqref{eq: abstract_model}, which provide the foundation for the comparative statics analysis in Section~\ref{sec: comparativestatics} and the solution algorithms developed in Section~\ref{sec: algorithms}.

\begin{proposition}
\label{prop:convexity_x}
    For a fixed $n \in \Bs{R}^{|\Cs{J}|}$, $F(\cdot, n)$, as defined in \eqref{eq: abstract_model}, is convex on $\Bs{R}^{|\Cs{I}|}$. 
\end{proposition}
\proof{Proof.} 
    Recall that $F(x, y) := e+ d^\top n + r^\top x  + \sum_{j \in \Cs{J}} f_j(x, n_j)$, where $f_j(x, n_j) = h_j(n_j;\sigma_j) \sum_{i \in \Cs{I}}  q_{ij} \Psi\left(\frac{g_{ij}(x)}{h_j(n_j;\sigma_j)}\right)$, as defined in \eqref{eq: abstract_model_f_j}. Hence, it suffices to show $f_j(\cdot, n_j)$ is convex for all $j \in \Cs{J}$. For simplicity, let us suppress $j$. Consider $K_i(x):=h(n;\sigma) \Psi\left(\frac{g_{i}(x)}{h(n;\sigma)}\right)$ for $i \in \Cs{I}$. Let $\alpha=h(n\sigma)>0$ and $r_i(x)=g_{i}(x)/ \alpha$. 
    Then, $K_i(x)=\alpha \Psi\big(r_i(x)\big)$. We claim that $K_i(\cdot)$ is convex for all $i \in \Cs{I}$, which leads to convexity of $f(\cdot,n)$ as $q_i>0$ for $i \in \Cs{I}$. We have  
    $\nabla K_i(x)  =\alpha 
\nabla r_i(x) \Psi'\big(r_i(x)\big)= 
\nabla g_i(x) \Phi\big(r_i(x)\big)$. Hence, 
$\nabla^2 K_i(x) = 
\nabla g_{i}(x) 
\nabla r_i(x)^\top \Phi'\big(r_i(x)\big) 
        = 
\nabla g_{i}(x) 
\nabla g_{i}(x)^\top \phi\big(r_i(x)\big) / \alpha 
        \succeq 0$.
\hfill \Halmos 
\endproof  

\begin{proposition}
\label{prop:convexity_n}
    For a fixed $x \in \Bs{R}^{|\Cs{I}|}$, $F(x,\cdot)$, as defined in \eqref{eq: abstract_model}, is strictly convex on $\Bs{R}_{+}^{|\Cs{J}|}$. 
\end{proposition}
\proof{Proof.} 
Similar to the proof of Proposition~\ref{prop:convexity_x}, it suffices to show 
$f_j(x, n_j)$ is strictly convex in $n_j$ on $\Bs{R}$ for each $j \in \Cs{J}$. For simplicity, let us suppress $x$, $j$, and $\sigma$. 
Consider $f(n):= h(n) \sum_{i \in \Cs{I}} q_i \Psi\big(r_i(n)\big)$, where $r_i(n):=\frac{g_i}{h(n)}$. Moreover, define
%\begin{equation*}
$A(n):=\sum_{i \in \Cs{I}}q_i\phi\big(r_i(n)\big)$ and $ B(n):=\sum_{i \in \Cs{I}}q_i r_i(n)^2\phi\big(r_i(n)\big)$.
%\end{equation*}
Using $\Psi'(z)=\Phi(z)$ and $\Psi(z)-z\Phi(z)=\phi(z)$, we have $f'(n)=h'(n)A(n)$ and 
%\begin{equation*}
$f''(n)=h''(n)A(n)+\frac{\big(h'(n)\big)^2}{h(n)}B(n)$.
%\end{equation*}
Since $q_i>0$, $\phi(\cdot)>0$, and $h(n)>0$, we have $A(n) >0$ and $B(n)>0$ for all $n \ge 0$. 
By Assumption~\ref{assum:properties}, $h'(n)<0$ and $h''(n) > 0$  for $n \ge 0$. Thus, $f''(n)>0$ for $n\ge 0$; hence, $f$ is strictly convex on $[0, \infty)$. 
\hfill \Halmos 
\endproof

%%%%%%%%%%%%%%%%%%%%%%%%%%%%%%%%%%%%%%%%%%%%%%%%%
\section{Comparative Statics Analysis}
\label{sec: comparativestatics}
%%%%%%%%%%%%%%%%%%%%%%%%%%%%%%%%%%%%%%%%%%%%%%%%%

This section develops local comparative statics for the two-stage stochastic FLP with endogenous demand learning, characterized by \eqref{eq: abstract_model}. Specifically, we examine how changes in the key learning parameters---the sampling cost, $d$, and the initial demand uncertainty, $\sigma$---affect the local minimizer and the corresponding local optimal value. Throughout the analysis, we assume that the optimal discrete facility decision remains unchanged within a neighborhood of the parameter values under consideration (see Assumption~\ref{assum:regularity}).

Observe that for a fixed $x \in \Cs{X}$, $f(x, \cdot)$, defined in \eqref{eq: abstract_model_f}, is additive over $j$. That is, $f(x, n):=\sum_{j \in \Cs{J}} f_j(x,n_j)$, where 
\begin{equation}
\label{eq: abstract_model_f_j}
    f_j(x, n_j) := h_j(n_j;\sigma_j) \sum_{i \in \Cs{I}}  q_{ij} \Psi\left(\frac{g_{ij}(x)}{h_j(n_j;\sigma_j)}\right).
\end{equation}
Given this separability, our goal in this section is to study the sensitivity of local minimizers of $F$, defined in \eqref{eq: abstract_model}, over $\Cs{X} \times \Cs{N}$ with respect to the scalar parameters $d_j$ and $\sigma_{j}$, $j \in \Cs{J}$, varying one parameter at a time while holding all other model parameters fixed.
To isolate the effect of parameter changes on coordinate $j$, we fix $x \in \Cs{X}$ and $n_{-j}:=(n_\ell)_{\ell 
\neq j}$. 

\begin{assumption}[Local Regularity Condition]
\label{assum:regularity}
Consider $j \in \Cs{J}$. For a reference parameter value $\bar\theta_j$, let $\big(x^\ast(\bar \theta_j),n^\ast(\bar \theta_j)\big)$ be a local minimizer of $F$, defined in \eqref{eq: abstract_model}, over $\Cs{X} \times \Cs{N}$. Suppose that
there exists an open neighborhood $\Cs{T}$ of $\bar\theta_j$ such that for every $\theta_j\in\Cs{T}$, the corresponding local minimizer $\big(x^\ast(\theta_j),n^\ast(\theta_j)\big)$ of $F$ over $\Cs{X} \times \Cs{N}$ satisfies $x^\ast(\theta_j)=x^\ast(\bar\theta_j)$, i.e., the optimal discrete decision remains unchanged locally, and $n^\ast(\theta_j) \in (0,b)$. 
\end{assumption}

Let $\theta_j\in \{d_j,\sigma_{j}\}$ denote the free parameter. 
Define the one-dimensional subproblem
\begin{equation}
\label{eq:subproblem}
\min_{0\le n_j\le b} \Bigg\{\varphi_j(n_j;\theta_j):=d_j n_j + f_j(x,n_j)\Bigg\}. 
\end{equation}
Under Assumption \ref{assum:regularity} and the strict convexity of $F(x,\cdot)$ on $\Bs{R}_{+}^{|\Cs{J}|}$, we have 
%\begin{equation*}
$\frac{\partial \varphi_j}{\partial n_j}(n_j;\theta_j)=0$ and $\frac{\partial^2 \varphi_j}{\partial n_j^2}( n_j;\theta_j)>0$
%\end{equation*}
at $n=n^\ast(\theta_j)$ for $\theta_j \in \Cs{T}$ (see Proposition~\ref{prop:convexity_n} and Corollary~\ref{cor:min-g}).
Thus, the implicit-function theorem implies that there exists a unique differentiable local solution mapping $n_j^\ast(\theta_j)$ on $\Cs{T}$. Define the resulting differentiable local optimal value function
\begin{equation*}
V^\ast_j(\theta_j):=\varphi_j\big(n_j^\ast(\theta_j);\theta_j\big).
\end{equation*}
%Observe that $V^\ast_j(\theta_j)$ is differentiable on $\Cs{T}$. 

\subsection{Main Theorems}
\label{sec: cs_results}
We are now ready to state the main theorems in this section. 

%, we introduce some notation (see Section \ref{secEC:compstat} for explicit definitions). Let $d_0(\theta_j)=-\frac{\partial h_j}{\partial n_j} (0)\sum_{i \in \Cs{I}} p_{ij} \phi\Big(\frac{g_{ij}(x)}{h_j(0; \theta_j)}\Big)$.

\begin{theorem}[Sensitivity of Local Minimizer]
\label{thm:cs_min}
Consider customer $j \in \Cs{J}$. Under Assumption \ref{assum:regularity}, for $\theta_j \in \Cs{T}$, $\theta_j \in \{d_j, \sigma_{j}\}$, we have 
$n^\ast_j$ locally  
\begin{enumerate}[label=(\roman*).]
\item decreasing in the sampling cost, i.e., 
$\frac{\partial n_j^\ast}{\partial d_j}<0$,
\item increasing in the initial demand uncertainty, i.e., $\frac{\partial n_j^\ast}{\partial \sigma_{j}} >0 $. 
\end{enumerate}
\end{theorem}

\begin{theorem}[Sensitivity of Local Optimal Objective Value]
\label{thm:cs_value}
Consider customer $j \in \Cs{J}$. Under Assumption \ref{assum:regularity}, for $\theta_j \in \Cs{T}$, $\theta_j \in \{d_j, \sigma_{j}\}$, we have 
$V^\ast_j$ locally increasing, i.e., $\frac{\partial V_j^\ast}{\partial \theta_{j}} >0 $. 
\end{theorem}

Theorems~\ref{thm:cs_min} and \ref{thm:cs_value} establish local sensitivity results for both the local minimizer $n_j^\ast$ and the local optimal value $V_j$ on regions where the optimal facility decision $x^\ast$ remains locally unchanged, as ensured by Assumption~\ref{assum:regularity}. 
We glean the following insights from the above results.
(i) The optimal learning effort decreases when the cost per unit of effort increases, i.e., more costly signals result in less optimal exploration (effort).
(ii) The optimal learning effort increases when the uncertainty in demand increases, i.e., higher uncertainty induces more exploration.
(iii) Increasing the sampling cost and demand variance will result in an increase in the total cost.

These results need not hold across parameter regions where the optimal discrete decision changes, since the solution mapping may become nonsmooth at such points. 
Likewise, even for the stochastic facility location problem \eqref{eq: FL_DI_first}, which does not include the sampling decision $n$, the sensitivity of the optimal discrete facility decision $x$ with respect to model parameters, such as the installation cost $u$, capacity, $D$, and transportation/penalty cost, $c$, is typically investigated numerically in the literature rather than through theoretical comparative statics analysis. 

\subsection{Proofs}
\label{sec:compstat_proof}

%In this section, we analyze the comparative statics of a local minimizer and the local optimal objective value of the subproblem \eqref{eq:subproblem} with respect to $\theta_j \in \{d_j, \sigma_{j}\}$. 
To simplify notation, we suppress the index $j$ in this section, whenever there is no ambiguity. We also drop the parameter argument $\theta$ when writing derivatives of the univariate functions $h(n;\theta)$ and $\varphi(n;\theta)$ with respect to $n$. 

Before proving main theorems, we first present the structural properties and derivatives of the objective function. 

%\subsection{Structural Properties of the Objective Function}
%\label{secEC:obj_structure}

\begin{lemma}[Structure of $\varphi(n)$ and its Derivatives]\label{lem:char-g}

Consider the function $\varphi(n)$, as defined in \eqref{eq:subproblem}, and its derivatives. 
Define $d_{0}:=-  h'(0)A(0)>0$, where $A(n):=\sum_{i \in \Cs{I}}q_i\phi\left(\frac{g_i(x)}{h(n)}\right)>0$. Then, 
\begin{enumerate}[label=(\roman*).]
  \item $\varphi$ is strictly convex on $[0,\infty)$. That is, $\varphi'$ is  strictly increasing on $[0,\infty)$. 
   \item $\varphi$ is strictly increasing on $[0,\infty)$ if and only if $d \ge d_{0}$. 
    \item $\varphi'$ has unique root $n_{min}\in(0,\infty)$ on $[0,\infty)$ if and only if $d<d_{0}$, with $\varphi$ strictly decreasing on $[0, n_{min})$ and strictly increasing on $(n_{min},\infty)$.
  
\end{enumerate}
\end{lemma}

\proof{Proof.}

We have $\varphi(n)= dn+ h(n) \sum_{i \in \Cs{I}} q_i \Psi\big(r_i(n)\big)$, where $r_i(n):=\frac{g_i(x)}{h(n)}$. Define $B(n):=\sum_{i \in \Cs{I}}q_i r_i(n)^2\phi\big(r_i(n)\big)$. Then, we have 
%Using $\Psi'(z)=\Phi(z)$ and $\Psi(z)-z\Phi(z)=\phi(z)$, we have 
$\varphi'(n)=d+h'(n)A(n)$ and 
%\begin{equation*}
$\varphi''(n)=h''(n)A(n)+\frac{(h'(n))^2}{h(n)}B(n)$. Similar to the proof of Proposition~\ref{prop:convexity_n}, we have 
%\end{equation*}
%Since $q_i>0$, $\phi(\cdot)>0$, and $h(n)>0$, we have $A(n) >0$ and $B(n)>0$ for all $y \ge 0$. 
%By Lemma~\ref{lem:sigma-signs}, $h'(n)<0$ and $h''(n) > 0$  for $y \ge 0$. Thus, 
$\varphi''(n)>0$ for $n\ge 0$; hence, $\varphi$ is strictly convex on $[0, \infty)$, proving part (i).

Now, note that $\varphi'(0)=d+h'(0)A(0)=d-d_0$. Moreover, we have $h'(n)\to0$ as $n\to\infty$. By Assumption~\ref{assum:properties}, there exists a constant $0<\bar{a}<+\infty$ such that $h(n)\to\bar a$. Since $A(n)\to\sum_{i \in \Cs{I}}q_i\phi(g_i(x)/\bar a)<\infty$ as $n\to\infty$, we obtain $\lim_{n\to\infty}\varphi'(n)=d$. 

For part (ii), if $d\ge d_0$, then $\varphi'(0)\ge0$. Since $\varphi'$ is strictly increasing by part (i), $\varphi'(n)>0$ for all $n>0$. Therefore, $\varphi$ is strictly increasing on $[0,\infty)$. Conversely, suppose that $\varphi$ is strictly increasing on $[0,\infty)$. Suppose by contradiction that  $d<d_0$. Then, $\varphi'(0)<0$. By continuity of $\varphi'$, there exists $\epsilon>0$ such that $\varphi'(n)<0$ for all $n\in[0,\epsilon]$, contradicting that $\varphi$ is strictly increasing on $[0,\infty)$. Hence, $\varphi$ is strictly increasing on $[0,\infty)$ if and only if $d\ge d_0$.

For part (iii), if $0<d<d_0$, then $\varphi'(0)<0$, whereas $\lim_{y\to\infty}\varphi'(n)=d>0$ as $\lim_{y\to\infty} h'(n)=0$ and $\lim_{y\to\infty} A(n)=0$. Continuity and strict monotonicity of $\varphi'$ imply that there exists a unique $n_{min}\in(0,\infty)$ such that $\varphi'(n_{min})=0$. 
Conversely, suppose that $\varphi'$ has a unique root $n_{min}\in(0,\infty)$. Since $\varphi'$ is strictly increasing, $0=\varphi'(n_{min})>\varphi'(0)=d-d_0$, so $d<d_0$. Also, strict monotonicity and $\lim_{n\to\infty}\varphi'(n)=d$ give $0=\varphi'(n_{min})<d$, so $d>0$. Therefore, $0<d<d_0$. The sign pattern follows from the strict monotonicity of $\varphi'$: $\varphi'(n)<0$ for $n<n_{min}$ and $\varphi'(n)>0$ for $n>n_{min}$.
\hfill \Halmos
\endproof

\begin{lemma}[Sufficient Conditions for Global Minimizer of $\varphi(n)$]\label{lem:min-g}
Consider the function $\varphi(n)$, as defined in \eqref{eq:subproblem}. 
Define $d_{0}:=-  h'(0)A(0)>0$, where $A(n):=\sum_{i \in \Cs{I}}q_i\phi\left(\frac{g_i(x)}{h(n)}\right)>0$.
Consider a global minimizer $n^\ast$ of $\varphi$ over $[0,\infty)$. Moreover, consider $n_{min}$, as defined in Lemma \ref{lem:char-g}.  
\begin{enumerate}[label=(\roman*).]
  \item $n^\ast=0$ is the unique global minimizer of $\varphi$  if and only if $d \ge d_0$. 
  \item $n^\ast=n_{min}$ is the unique global minimizer of $\varphi$ if and only if $0<d<d_0$.  
\end{enumerate} 
\end{lemma}
\proof{Proof.}
By Lemma \ref{lem:char-g}.(ii), $\varphi$ is strictly increasing on $[0, \infty)$ if and only if $d \ge d_0$. Thus, $n^\ast=0$ is the unique global minimizer of $\varphi$ over $[0,\infty)$.
On the other hand, by Lemma \ref{lem:char-g}.(iii), $\varphi'$ has a unique root $n_{min}$ on $[0,\infty)$ if and only if $0<d< d_0$. As $\varphi$ is strictly convex on by Lemma \ref{lem:char-g}.(i), $n_{min}$ is the unique global minimizer over $[0,\infty)$ if and only if $0<d< d_0$.
\hfill \Halmos
\endproof

% \begin{remark}
%     Lemma \ref{lem:min-g} implies that under Assumption~\ref{assum:regularity}, $\theta \in \Cs{T}$ at which $n^\ast(\theta)$ is an interior point, i.e., $n^\ast(\theta) \in(0,b)$, satisfies  $0<d < d_0$. 

% \end{remark}

\begin{corollary}[Necessary Conditions for Global Minimizer of $\varphi(n)$]\label{cor:min-g}
Consider the function $\varphi(n)$, as defined in \eqref{eq:subproblem}. 
Consider a global minimizer $n^\ast$ of $\varphi$ over $[0,\infty)$. Moreover, consider $n_{min}$, as defined in Lemma \ref{lem:char-g}.  
If $n^\ast=0$, we have $\varphi'(0) y \ge 0$ for every $y \in [0,\infty)$ and and $\varphi$ satisfies a local linear growth around $n^\ast$. If $n^\ast=n_{min}$, we have $\varphi'(n_{min})=0$, $\varphi''(n_{min})>0$, and $\varphi$ satisfies a local quadratic growth around $n^\ast$. 
\end{corollary}

\proof{Proof.}
Suppose $n^\ast=0$. By Lemma \ref{lem:char-g}.(i)--(ii), $\varphi'(0) >0$  and $\varphi''(0)>0$. Therefore,  $\varphi'(0) n \ge 0$ for every $n \in [0,\infty)$. Moreover, because $\varphi'$ is continuous, there exists $b>\varepsilon>0$ and $\hat n \in [0,\varepsilon)$ such that  $\varphi'(\hat n)\ge \varphi'(0)/2$. Therefore, by the mean value theorem,  $\varphi(n) \ge \varphi(0) + \frac{\varphi'(0)}{2} n$ for $n \in  [0,\varepsilon)$, implying a local linear growth around $n^\ast=0$. 

Now, suppose $n^\ast=n_{min}$. By Lemma \ref{lem:char-g}.(iii), $\varphi'(n_{min})=0$ and by Lemma \ref{lem:char-g}.(i), $\varphi''(n_{min})>0$. Moreover, because $\varphi''$ is continuous, there exists $b>\varepsilon>0$ and $\hat n \in \Bs{N}_{\varepsilon}(n_{min})$ such that  $\varphi''(\hat n)\ge \varphi''(n_{min})/2$.  Therefore, by the mean value theorem,  $\varphi(n) \ge \varphi(n_{min}) + \frac{\varphi''(n_{min})}{2} (n-n_{min})^2$ for $n \in \Bs{N}_{\varepsilon}(n_{min}) \cap [0,b]$, implying a local quadratic growth around $n^\ast$. 
\hfill \Halmos
\endproof

%\subsection{Partial Derivatives of Objective Function}
%\label{secEC:obj_derivative}

\begin{lemma}[Signs of Partial Derivatives of the Objective Function]\label{lem:g-signs}
%\begin{enumerate}[label=(\roman*).]
    %\item 
    (i) $\frac{\partial \varphi}{\partial \sigma}(n;\sigma)>0$ for all $n \ge 0$, and 
    %\item 
    (ii) $\frac{\partial^2 \varphi}{\partial n \partial \sigma}(n;\sigma)<0$ for all $n\ge 0$. 
%\end{enumerate}
\end{lemma}
\proof{Proof.}
We have $\varphi(n;\theta)= dn+ h(n;\sigma) \sum_{i \in \Cs{I}} q_i \Psi\big(r_i(n;\sigma)\big)$, where $r_i(n;\sigma):=\frac{g_i(x)}{h(n;\sigma)}$. Define
%\begin{equation*}
$A(n;\sigma):=\sum_{i \in \Cs{I}}q_i\phi\big(r_i(n;\sigma)\big)$. % and $B(n;\theta):=\sum_{i \in \Cs{I}}q_i r_i(n;\theta)^2\phi\big(r_i(n;\theta)\big)$. 
The first partial derivative with respect to $\sigma$ is
%\begin{equation*}
$\frac{\partial \varphi}{\partial \sigma}(n;\sigma)
=\frac{\partial h}{\partial \sigma}(n;\sigma)\sum_{i \in \Cs{I}}q_i\Big[\Psi\big(r_i(n;\sigma)\big)-r_i(n;\sigma)\Phi\big(r_i(n;\sigma)\big)\Big]
=\frac{\partial h}{\partial \sigma}(n;\sigma) A(n;\sigma)$.
%using $\Psi'(z)=\Phi(z)$ and $\Psi(z)-z\Phi(z)=\phi(z)$. 
By direct differentiation, we obtain
%\begin{equation*}
$\frac{\partial^2 \varphi}{\partial n \partial \sigma}(n;\sigma)
= \frac{\partial^2 h}{\partial n \partial \sigma}(n;\sigma)A(n;\sigma)
+ \frac{\frac{\partial h}{\partial n}(n;\sigma)\frac{\partial h}{\partial \sigma}(n;\sigma)}{h(n;\sigma)^3}
\sum_{i \in \Cs{I}} q_i g_i(x) \phi \big(r_i(n;\sigma)\big)$.
%\end{equation*}
Note that $\phi(\cdot)>0$, $h(\cdot)>0$, and $q_i>0$ and $g_i(x)\ge0 $ for all $i \in \Cs{I}$. Hence, the sign of $\frac{\partial \varphi}{\partial \sigma}(n;\sigma)$ equals the sign of $\frac{\partial h}{\partial \sigma}(n;\sigma)$; which is positive by Assumption~\ref{assum:properties}. This assumption also ensures that 
$\frac{\partial h}{\partial n}(n;\sigma)<0$ and $\frac{\partial h^2}{\partial n \partial \sigma}(n;\sigma)<0$ for $n\ge 0$, which prove (ii). 
\hfill \Halmos
\endproof

%\subsection{Proofs of Theorems}
%\label{secEC:cs_proofs}
We are now ready to prove the main theorems. 

\proof{Proof of Theorem \ref{thm:cs_min}.}
The first-order condition is $\frac{\partial \varphi}{\partial n}(n;\theta)=0$ at $n=n^\ast(\theta)$. 
Differentiating with respect to $\theta$ gives
%\begin{equation*}
$\frac{\partial^2 \varphi}{\partial n^2}(n;\theta) \frac{\partial n^\ast}{\partial\theta}
+\frac{\partial^2\varphi}{\partial n \partial \theta}( n;\theta)=0$
%\end{equation*}
at $n=n^\ast(\theta)$.
Hence,
%\begin{equation*}
$\frac{\partial n^\ast}{\partial\theta}
=-\frac{\partial^2 \varphi}{\partial n \partial \theta}(n;\theta)/{\frac{\partial^2 \varphi}{\partial n^2}( n;\theta)}$.
%\end{equation*}
Recall that $\frac{\partial^2 \varphi}{\partial n^2}( n;\theta)>0$ at $n=n^\ast(\theta)$ by Corollary~\ref{cor:min-g}. 
For $\theta=d$, $\frac{\partial n^\ast}{\partial d}<0$  
since $\frac{\partial^2 \varphi}{\partial n \partial d}(n;d)=1$. This proves part (i). For $\theta=\sigma$, the sign of $\frac{\partial n^\ast}{\partial\theta}$ is determined by the opposite sign of the corresponding cross-partial derivative, as given by Lemma~\ref{lem:g-signs}, proving part (ii).
\hfill \Halmos
\endproof

\proof{Proof of Theorem \ref{thm:cs_value}}
By definition, $V^\ast(\theta):=\varphi(n^\ast(\theta);\theta)$.
Using the chain rule, we have 
%\begin{equation*}
$\frac{\partial V^\ast}{\partial \theta}
=\frac{\partial \varphi}{\partial n}(n; \theta)\frac{\partial n^\ast}{\partial \theta}
+\frac{\partial \varphi}{\partial \theta}(n; \theta)$
%\end{equation*}
at $n=n^\ast(\theta)$.
Recall that by the first-order condition, we have $\frac{\partial \varphi}{\partial \theta}\big(n^\ast(\theta);\theta\big)=0$. Hence, $\frac{\partial V^\ast}{\partial \theta}
=\frac{\partial \varphi}{\partial \theta}(n; \theta)
$. 
In particular, $\frac{\partial V^\ast}{\partial d}= y >0$ since $\frac{\partial \varphi}{\partial d}=y$. For $\theta= \sigma$, the sign of $\frac{\partial \varphi}{\partial \sigma}(n; \sigma)$ is positive by Lemma~\ref{lem:g-signs}. This completes the proof. 
\hfill \Halmos
\endproof

%%%%%%%%%%%%%%%%%%%%%%%%%%%%%%%%%%%%%%%%%%%%%%%%%
\section{Solution Algorithms}
\label{sec: algorithms}
%%%%%%%%%%%%%%%%%%%%%%%%%%%%%%%%%%%%%%%%%%%%%%%%%

The nonconvex function $f(x,n)$, defined in \eqref{eq: abstract_model_f}, poses computational challenges for solving the nonconvex MIP formulation \eqref{eq: abstract_model}, whose objective function is $F(x,n)$. Despite the nonconvexity of $f$, the objective exhibits blockwise convexity properties: $F(\cdot,n)$ is convex for every $n \in \Cs{N}$ (Proposition \ref{prop:convexity_x}), while $F(x,\cdot)$ is strictly convex for every $x \in \Cs{X}$ (Proposition \ref{prop:convexity_n}). 
We exploit these unique  properties to develop an alternating optimization procedure. However, the nonlinear structure of $f$, primarily due to the function $\Psi$, defined in \eqref{eq:Psi}, and its product with $h(n,\;\sigma)$, still remains computationally challenging.
 
In this section, we develop solution algorithms to solve \eqref{eq: abstract_model} based on (linearizable) approximations of $f(x,n)$ and, consequently, of the objective function $F(x,n)$. 
Exploiting the finiteness of the discrete facility decisions $x$ and the separability of $F$ in the continuous sampling decisions $n$ for a fixed $x$, we design {\it proximal} block coordinate descent algorithms. We prove that the resulting iterates converge to a local minimizer of $F$ and establish the \emph{consistency} of the corresponding sets of local minimizers/stationary points.
We consider two approximations to the function $f(x,n)$: 
\begin{itemize}
    \item {\it sample average approximation} (SAA) using {\it one} set of $M$ realizations from $\textrm{N}(0,1)$ to handle $\Psi\left(\frac{g_{ij}(x)}{h_j(n_j;\sigma_j)}\right)$ for all $i \in \Cs{I}$ and $j \in \Cs{J}$, given that $\Psi(z)=\e{(\zeta+z)_{+}}$, where $\zeta \sim \textrm{N}(0,1)$ and $(a)_{+}=\max\{0,a\}$ (see Lemma \ref{lem:SAA}), 
    \item {\it piece-wise linear approximation} (PWLA) using {\it one} set of $M$ breakpoints to interpolate $\Psi\left(\frac{g_{ij}(x)}{h_j(n_j;\sigma_j)}\right)$ for all $i \in \Cs{I}$ and $j \in \Cs{J}$. 
\end{itemize}
These approximations may include some (mixed-integer) auxiliary variables and (linear) constraints. Nevertheless, given a fixed $x$, $n$, and $M$, the auxiliary variables are uniquely determined by the constraints. Hence, we suppress them in our notation. % and simply denote such an approximation by $K(x,n;M)$. 

We are now ready to present the proposed algorithm to solve \eqref{eq: abstract_model}, or the resulting problems using either SAA or PWLA, as outlined in Algorithm~\ref{alg:bcd}. 
Given an initial point $n^0$, the algorithm iterates through the following steps:

{\bf $x$-step: MIP Update for $x$}:
Solve for $x^{k+1}$ given $n^k$:
\begin{equation}
\label{eq:x_step}
    x^{k+1} = T_M(n^k) \in  \argmin_{x \in \Cs{X}} P(x, n^k;M),
\end{equation}
where $P(x, n;M)$ is an approximation of $F(x,n)$, parametrized by $M \in \Bs{N}$, with the property that $P(\cdot, n;M)$ is convex for every $n \in \Cs{N}$ and $M \in \Bs{N}$. 
Thus, the $x$-step is a pure-integer convex program. We break the tie with a Lexicographical ordering, e.g., $x^1 \in \Cs{X}$ is preferred over $x^2 \in \Cs{X}$ if 
there exists an index $i^\ast \in \Cs{I}$ such that $x^1_{i^\ast}=1$ and $x^2_{i^\ast}=0$, with $|x^1_{i}-x^2_i|=0$ for all $i<i^\ast$. That is, $T_M(n^k): \Cs{N} \to \Cs{X}$ is a single-valued mapping (see Assumption~\ref{assum:stability}). 
%Function $P(x,n;M)$ is defined as $P(x,n;M)= \min_{u \in \Cs{U}(x,n)} \ K(x,n,u;M)$, where $\Cs{U}(x,n)$ indicates the tie between auxiliary decision variables $u$ and the original decision variables $x$ and $y$. It is easy to verify that if for a given $y$ and $M$, $K(\cdot,n,\cdot;M)$ is convex (jointly in $x$ and $u$) and $\Cs{U}(x,n)$ is a convex set, then $P(\cdot,n;M)$ is convex. 

{\bf $n$-step: Proximal Continuous Update for $n$}:
Solve for $n^{k+1}$ given $x^{k+1}$:

\begin{equation}
\label{eq: n_step}
    n^{k+1} \in \argmin_{n \in \Cs{N}}  \left\{ Q(x^{k+1}, n;M) + \frac{\lambda_k}{2} \|n - n^k\|^2 \right\},
\end{equation}
where $\lambda_k > 0$ and $Q(x, n;M)$ is an approximation of $F(x,n)$, parametrized by $M \in \Bs{N}$, with the property that $Q(x,\cdot;M)$ is convex (not necessarily strictly convex) and additive for every $x \in \Cs{X}$ and $M \in \Bs{N}$, i.e., $Q(x, n;M):=\sum_{j \in \Cs{J}} Q_j(x, n_j;M)$. 
Hence, we can simply write the $n$-step as 
\begin{equation*}
    n^{k+1}_j \in \argmin_{n_j \in [0,b]} \left\{ Q_j(x^{k+1}, n_j;M) + \frac{\lambda_k}{2} (n_j - n^k_j)^2 \right\},
\end{equation*}
for each $j \in \Cs{J}$. 

We make the following local stability assumption about the singled-valued mapping $T_{M}$. 
\begin{assumption}[Local Stability Condition]
    \label{assum:stability}
    If the sequence $\{n^k\}$ converges to $\bar n$, the single-valued mapping $T_{M}$ is locally stable at $\bar n$; that is, there exist $\bar x \in\Cs{X}$ and sufficiently large $K$ such that for all $k > K$, we have $\bar x = T_M(n^k) \in \argmin_{x \in \Cs{X}} P(x,n^k;M)$. % for all $n\in\Cs{N}\cap\Bs{N}_\rho(\bar n)$. 
\end{assumption}
%Assumption~\ref{assum:properties} implies that  $x^{k+1}=x^\ast$ for all sufficiently large $k$. 
\begin{remark}
\label{rem:stability}
By the continuity of $P(z,\cdot;M)$ and finiteness of $\Cs{X}$, Assumption~\ref{assum:properties} holds if   $\bar x$ is the unique minimizer of $P(\cdot,\bar n;M)$ over $\Cs{X}$. Local stability also holds
when multiple minimizers remain tied in a neighborhood of $\bar n$ and the
tie-breaking rule selects the same $x$ throughout that neighborhood.
\end{remark}

Theorem~\ref{thm:min_convergence} establishes convergence of Algorithm \ref{alg:bcd}, when both $P(\cdot;M)$ and $Q(\cdot; M)$ are the same as $F$; this is an algorithm to solve the {\it original} problem \eqref{eq: abstract_model}, where the optimization problem in each block is nonlinear. 

\begin{theorem}[Convergence to a Local Min]
\label{thm:min_convergence}
    In Algorithm \ref{alg:bcd}, suppose that $P(\cdot;M)=Q(\cdot;M)=F$, as defined in \eqref{eq: abstract_model}. Suppose that the sequence $\{n^k\}$ is convergent. 
    Then, under Assumption~\ref{assum:stability}, Algorithm \ref{alg:bcd} converges to a local minimum of $F$ over $\Cs{X} \times \Cs{N}$. 
\end{theorem}

When $Q(\cdot;M)=F$, the corresponding $n$-step does not require a proximal term to guarantee sufficient decrease (see Lemma~\ref{lem:sifficient_decrease}), as $F(x,\cdot)$ is strictly convex for a fixed $x$ (see Proposition~\ref{prop:convexity_n}). 
Consequently, because both the $x$- and $n$-steps lead to unique solutions, standard block coordinate descent results guarantee convergence to a stationary point, which is also a local minimum; see, e.g., \cite{zangwill1969nonlinear} and \citet[Proposition~3.7.1]{bertsekas2016nonlinear}. 
In the remainder of this section, we develop two variants of Algorithm~\ref{alg:bcd}. To avoid solving nonlinear optimization problems in each block, we utilize linearizable approximations: (i) Both $P(\cdot;M)$ and $Q(\cdot;M)$ are formed by SAA (Section~\ref{sec: alg_SAA}) and (ii) $P(\cdot;M)$ is formed by PWLA and $Q(\cdot; M)=F$ (Section~\ref{sec: alg_PWL}). %, referred to as {\it PWLA}. 

% \begin{remark}
% Suppose that $Q(x, y;M)=F(x,n)$ for all $(x,n)$ and $M \in \Bs{N}$. 
% Using $\Psi'(z)=\Phi(z)$, the derivative of $F(x,n)$ with respect to $n_j$, $j \in \Cs{J}$, is
% \begin{equation*}
%     \frac{\partial F_j}{\partial n_j}(x,n_j)= d_j + h_j'(n_j;\sigma_j) \sum_{i \in \Cs{I}}  q_{ij} \phi \left(\frac{g_{ij}(x)}{h_j(n_j;\sigma_j)}\right). 
% \end{equation*}
% Hence, checking the first-order and boundary points yields the solution to the $n$-step.
% Otherwise, the $n$-update can be implemented via the projected (univariate) proximal-gradient step    $n^{k+1}_j = \proj_{[0,b]} \left( n^k_j - \frac{1}{\lambda_k}  \frac{\partial Q_j}{\partial n_j} (x^{k+1},n^k_j;M)\right)$ 
% for each $j \in \Cs{J}$. 
% Alternatively, one can use backtracking on $\lambda_k$. Choose $\hat n_j=\proj_{[0,b]} \left( n^k_j - \frac{1}{\lambda_k}  \frac{\partial Q_j}{\partial n_j} (x^{k+1},n^k_j;M)\right)$, 
% and if $Q_j(x^{k+1},\hat n_j;M) > Q_j(x^{k+1},n^k_j;M) - \frac{\lambda_k}{2} (\hat n_j-n^k_j)^2$
% for some $j \in \Cs{J}$, then update $\lambda_k \leftarrow 2\lambda_k$ and repeat until the sufficient decrease condition holds.
% \end{remark}

\begin{algorithm}[!tb] 
    \SetAlgoLined
        \KwIn{$n^{0} \in [0, b]^{|\Cs{J}|}$, tolerance $\epsilon > 0$, $k = 0$.}
	    \KwOut{Local min $(x^\ast,n^\ast)$.} 
    \While{not converged}
    {
    \textbf{1. Discrete Block Update ($x$-step):}
    Solve the mixed-integer convex program \eqref{eq:x_step} to global optimality, with tie broken with a Lexicographical ordering to get 
    $x^{k+1}$. % = T(n^k) \in \argmin_{x \in \Cs{X}} P(x,n^{k}; M)$. 
    
    \textbf{2. Proximal Continuous Block Update ($n$-step):}
    Solve the convex program \eqref{eq: n_step} to get 
    $n^{k+1}$. % \in \argmin_{n \in \Cs{N}}  \left\{ Q(x^{k+1}, n; M) + \frac{\lambda_k}{2} \|n - n^k\|^2 \right\}$.
    
    %$\lambda_{k+1} \leftarrow \lambda_k$. 

    \textbf{3. Convergence Check:}
    \If{$\|n^{k+1} - n^{k}\| < \epsilon$ and $x^{k+1} = x^{k}$}{ 
        \Return $(x^\ast,n^\ast)=(x^{k+1}, n^{k+1})$
    }
    $k \leftarrow k + 1$
    }
\caption{Proximal Block Coordinate Descent to Solve \eqref{eq: abstract_model}}
\label{alg:bcd}
\end{algorithm}

\subsection{Algorithm for the Sample Average Approximation}
\label{sec: alg_SAA}

As noted, $\Psi(z)=\e{(\zeta+z)_{+}}$, where $\zeta \sim \textrm{Normal}(0,1)$. % (see Lemma~\ref{lem:SAA}).
Therefore, one can approximate this function with a sample average of $M$ realizations. 
Consider the SAA of \eqref{eq: abstract_model}
\begin{equation}
\label{eq: abstract_SAA}
    \min_{x \in \Cs{X}, n \in \Cs{N}} \Big\{F_{M}(x, n) := e+ d^\top n + r^\top x + f_{M}(x, n)\Big\}
\end{equation}
where 
%\begin{equation*}
    $f_{M}(x, n) := \frac{1}{M}\sum_{t=1}^M \sum_{j \in \Cs{J}} h_j(n_j;\sigma_j) \sum_{i \in \Cs{I}}  q_{ij} \left(\zeta^t  + \frac{g_{ij}(x)}{h_j(n_j;\sigma_j)}\right)_{+}$,
%\end{equation*}
with $\zeta^t$, $t=1,\ldots,M$ being i.i.d. realizations from $\textrm{N}(0,1)$. 
%, and 
%\begin{equation*}
%    \ell(\xi^t;x,n)=  = \sum_{j \in \Cs{J}} \sum_{i \in %\Cs{I}} \left(\xi^t h_j(n_j) + g_{ij}(x)\right)_{+}.
%\end{equation*}
Note that $f_M(x,n)$ is an unbiased estimator of $f(x,n)$. %=\e{\ell(\xi;x,n)}$, where $\xi \sim \textrm{Normal}(0,1)$. 

\begin{theorem}
\label{thm:min_convergence_SAA}
    In Algorithm \ref{alg:bcd}, suppose that $P(\cdot;M)=Q(\cdot;M)=F_M$, as defined in \eqref{eq: abstract_SAA}. Suppose that the sequence $\{n^k\}$ converges to $n^\ast_M$, and $x_M^\ast$ be a stability point under Assumption~\ref{assum:stability}. Moreover, assume that $F_M(x_M^\ast, \cdot)$ satisfies a local quadratic growth at $n_M^\ast$ on $\Cs{N}$. %that there exist constants $\delta_M>0$ and $c_M>0$ such that
    %\begin{equation*}
    %$ F_M(x_M^\ast,n) \ge F_M(x_M^\ast,n_M^\ast) + \frac{c_M}{2}\|n-n_M^\ast\|^2 \qquad \forall n\in \Cs{N}\cap \Bs{N}_{\delta_M}(n_M^\ast)$.
    %\end{equation*}
    Then, Algorithm \ref{alg:bcd} converges to $(x_M^\ast,n_M^\ast)$, a local minimum of $F_M$ over $\Cs{X} \times \Cs{N}$ with probability $1$. 
\end{theorem}
Theorem \ref{thm:saa_stationary_consistency} states that any stationary point of the SAA problem \eqref{eq: abstract_SAA} converges (along subsequences) to a stationary point of the original problem \eqref{eq: abstract_model} as $M\to\infty$. %, under a quadratic growth condition for the original objective function, $F$. 
If, moreover, the sampled objective function $F_M$ satisfies a local quadratic growth condition, the result can be stated for local minimizers. 
\begin{theorem}[Consistency of Stationary Points/Local Minima under SAA]
\label{thm:saa_stationary_consistency}
Consider 
%\begin{equation*}
$\Cs{S}:=\Big\{(x,n)\in \Cs{X}\times \Cs{N}: 0\in \partial_n\big(F(x,\cdot)+i_{\Cs{N}}\big)(n), \ x \in \argmin_{z\in\Cs{X}} F(z,n)\Big\}$  
%\end{equation*}
and 
%\begin{equation*}
$\Cs{S}_M:=\Big\{(x,n)\in \Cs{X}\times \Cs{N}:0\in \partial_n\big(F_M(x,\cdot)+i_{\Cs{N}}\big)(n), \ x \in \argmin_{z\in\Cs{X}} F_M(z,n)\Big\}$, 
%\end{equation*}
the set of (block) stationary points of \eqref{eq: abstract_model} and \eqref{eq: abstract_SAA}, respectively. 
Then, every sequence $(x_M,n_M)\in \Cs{S}_M$ admits a subsequence converging to some $(x^\ast,n^\ast)\in \Cs{S}$, and $\sup_{(x,n)\in\Cs{S}_M}\dist\big((x,n),\Cs{S}\big)\to 0$ as $M\to\infty$. If, moreover, for $(x^\ast,n^\ast)\in\Cs{S}$, $x^\ast$ is the unique minimzer of $F(\cdot,n^\ast)$, 
the Hausdorff distance satisfies
%\begin{equation*}
$ d_H(\Cs{S}_M,\Cs{S})\to 0$ as $M\to\infty$,
%\end{equation*}
with probability $1$. 
If for every $(x_M,n_M)\in \Cs{S}_M$,  $F_M(x_M,\cdot)$ satisfies a local quadratic growth condition at $n_M$ on $\Cs{N}$, then $\Cs{S}$ and $\Cs{S}_M$ represent the set of local minimizers of $F$ and $F_M$, respectively.  
\end{theorem}

\subsection{Algorithm for the Piece-wise Linear Approximation}
\label{sec: alg_PWL}
Consider a set of $M$ breakpoints $z^1,\ldots, z^M$, where $\underline{u}=z^0 <z^1<\cdots< z^M=\overline{u}$, such that they cover an interval $[\underline{u},\overline{u}]$
containing
$\mathcal Z:=\left\{\frac{g_{ij}(x)}{h_j(n_j;\sigma_j)}:
x\in\Cs{X},\ n\in\Cs{N}, \; i\in\Cs{I}, \; j\in\Cs{J}\right\}$. Define $\kappa_M:=\max_{t\in\{0,\dots,M-1\}}(z^{t+1}-z^t)$, and suppose that $\kappa_M\to 0$ as $M\to\infty$. 
Let $\Psi_M$ denote the piecewise-linear approximation of $\Psi$ %, defined in \eqref{eq:Psi}, 
built on these breakpoints.
For a fixed $n\in\Cs{N}$, one can overestimate $F(x,n)$ with function $S_M(x,n)$ as   
\begin{equation}
    \label{eq: pwl}
    S_M(x,n):= e+ d^\top n + r^\top x +  \sum_{j \in \Cs{J}} h_j(n_j;\sigma_j) \sum_{i \in \Cs{I}}  q_{ij} \Psi_M\left(\frac{g_{ij}(x)}{h_j(n_j;\sigma_j)}\right),
\end{equation}
while for a fixed $x\in\Cs{X}$, use the objective $F(x,\cdot)$, yielding a hybrid framework. 
Specifically, we consider a variant of Algorithm~\ref{alg:bcd} in which $P(\cdot;M)=S_M$, as defined in \eqref{eq: pwl}, and $Q(\cdot;M)=F$, as defined in \eqref{eq: abstract_model}. Similar to the argument following Theorem~\ref{thm:min_convergence}, the $n$-step does not require a proximal term to guarantee convergence. %Theorem~\ref{thm:pwl_stationary_consistency} establishes the consistency of the stationary points.

\begin{theorem}[Consistency of Stationary Points under PWLA]
\label{thm:pwl_stationary_consistency}
Let 
%\begin{equation*}
   $\Cs{S}:=\Big\{(x,n)\in \Cs{X}\times \Cs{N}: 0\in \partial_n\big(F(x,\cdot)+i_{\Cs{N}}\big)(n),  \ x \in \argmin_{z\in\Cs{X}} F(z,n)\Big\}$
%\end{equation*}
and 
%\begin{equation*}
$\Cs{S}_M:=\Big\{(x,n)\in \Cs{X}\times \Cs{N}: 0\in \partial_n\big(F(x,\cdot)+i_{\Cs{N}}\big)(n), \ x \in \argmin_{z\in\Cs{X}} S_M(z,n)\Big\}$,   
%\end{equation*}
denote the set of (block) stationary points of \eqref{eq: abstract_model} and hybrid stationary points induced by $P(\cdot;M)=S_M$, as defined in \eqref{eq: pwl}, and $Q(\cdot;M)=F$, as defined in \eqref{eq: abstract_model}, in Algorithm \ref{alg:bcd}. 
%Assume that for every $(x^\ast,n^\ast)\in \Cs{S}$, the function $F(x^\ast,\cdot)$ satisfies a quadratic growth condition at $n^\ast$ on $\Cs{N}$.
Then, every sequence $(x_M,n_M)\in \Cs{S}_M$ admits a subsequence converging to some $(x^\ast,n^\ast)\in \Cs{S}$, and 
and $\sup_{(x,n)\in\Cs{S}_M}\dist\big((x,n),\Cs{S}\big)\to 0$ as $M\to\infty$. If, moreover, for $(x^\ast,n^\ast)\in\Cs{S}$, $x^\ast$ is the unique minimzer of $F(\cdot,n^\ast)$, 
the Hausdorff distance satisfies
%\begin{equation*}
$d_H(\Cs{S}_M,\Cs{S})\to 0$ as $M\to\infty$.
%\end{equation*}
\end{theorem}
%\subsection{Extension}
% Suppose that 
% $y \in \Cs{N} = [0, b]^m \subset \Bs{R}^m$, and
% \begin{equation*}
%     f(x, y) = \sum_{j=1}^{m} h_j(n_j) \Psi\left(\frac{g_j(x)}{h_j(n_j)}\right)
% \end{equation*}
% For a fixed $x$, $f(x, y)$ is separable, i.e., $f(x, y)=\sum_{j=1}^{m} f_j(x,n_j)$. Hence, the $y$-step is separable, and we can extend the algorithm to this setting without any complication. Moreover,  the $y$-step can be implemented in parallel. This setting arises when we can uncorrelated prior.  

Theorem \ref{thm:pwl_stationary_consistency} states that any stationary point of the (hybrid) PWLA problem converges (along subsequences) to a stationary point of the original problem \eqref{eq: abstract_model} as $M\to\infty$. 

%%%%%%%%%%%%%%%%%%%%%%%%%%%%%%%%%%%%%%%%%%%%%%%%%\
\section{Numerical Results}
\label{sec: results}
%%%%%%%%%%%%%%%%%%%%%%%%%%%%%%%%%%%%%%%%%%%%%%%%%

In this section, we present the numerical study for the stochastic FLP with demand learning introduced in Section~\ref{sec: FL_Learning}. We explain the experiment design in Section~\ref{sec: design}. %how we generate the test instances and how deterministic benchmark data are converted into stochastic instances for our model. 
We then report the value of learning and comparative statics in Sections~\ref{sec: VoL} and \ref{sec: cs}, respectively. Finally, we present computational comparative results to test the efficacy of the proposed solution algorithms in Section~\ref{sec: comp}. %Then, we solve the closed-form formulation discussed in Proposition~\ref{prop: recourse_j} using the algorithms proposed in Section~\ref{sec: algorithms}, especially the SAA and PWL approaches in Sections~\ref{sec: alg_SAA} and~\ref{sec: alg_PWL}. Finally, we report the value of learning (VoL), discuss the main insights, and compare the performance of the proposed methods.

\subsection{Experimental Design}
\label{sec: design}

To conduct experiments, we adapt deterministic facility location benchmark instances 
introduced by \citet{holmberg1999exact}. These benchmark instances do not include several key parameters required by \eqref{eq: FL_DI_first_learning}, including demand uncertainty and learning-related parameters, as well as penalty cost for unmet demand. Therefore, we transformed benchmark instances ``p41", ``p51", and ``p56" into a form that is compatible with \eqref{eq: FL_DI_first_learning}. The details of this procedure are described below.

For each  $j \in \Cs{J}$, we set $\mu_j$, the mean of the demand distribution, as the (known) average demand of customer $j$. 
We also set $\sigma_j$, the initial demand uncertainty, as $\sigma_j=\sqrt{\alpha \mu_j}$. 
Moreover, we set $h_j(n_j;\sigma_j)=\sigma_j/\sqrt{1+\omega n_j}$, which satisfies Assumption~\ref{assum:properties} with $\omega >0$. 
Unless otherwise stated, we set the sampling cost $d_j=1$ for all $j \in \Cs{J}$.
For %The benchmark instances do not provide penalty costs for unmet demand. To make unmet demand more expensive than serving demand, we set the penalty cost for 
customer $j \in \Cs{J}$, we set \(
c_{0j} = \max_{i \in \mathcal{I}} c_{ij} +10\), where %$\delta > 0$ is a fixed margin and 
$c_{ij}$, $i \in \Cs{I}$ and $j \in \Cs{J}$, are transportation costs. For each facility $i \in \Cs{I}$ and customer $j \in \Cs{J}$, we set the installation cost $u_i =  \eta a_i^u \frac{1}{|\Cs{I}||\Cs{J}|}\sum_{i \in \Cs{I}}\sum_{j \in \Cs{J}} c_{ij} $ and capacity $D_{ij} = \rho a_i^D  \frac{1}{|\Cs{J}|}\sum_{j \in \Cs{J}}\mu_j$, where facility-specific parameters $a_i^u$ and $a_i^D$ were chosen such that facilities with larger capacities generally have higher installation costs.
% For each facility $i \in \Cs{I}$ and customer $j \in \Cs{J}$, we set installation cost $u_i=u$, where $u=\eta \frac{1}{|\Cs{I}||\Cs{J}|}\sum_{i \in \Cs{I}}\sum_{j \in \Cs{J}} c_{ij} $, and capacity $D_{ij}=D$, where $D=\rho \frac{1}{|\Cs{J}|}\sum_{j \in \Cs{J}}\mu_j$. 
% We chose the number of installed facilities to be $\left\lfloor |\Cs{I}|/2 \right\rfloor$. 
The parameters $\alpha,\omega, \rho$ and $\eta$ are all positive and tunable parameters, where $\alpha$ controls the initial uncertainty, $\sigma_j$, while $\omega$ controls the rate at which learning reduces uncertainty. Unless otherwise stated, $\alpha=3$ and $\omega=1$. 

All experiments are implemented in Python 3.13.5. Gurobi 13.0.1 is used as the MIP solver for location subproblems, while SciPy is used as a convex solver for one-dimensional smooth/nonsmooth learning subproblems. %, including univariate nonlinear minimization and root-finding. 
For benchmark comparison, we also use Knitro 15.0.0 as a commercial nonconvex solver. % to solve the joint formulation directly. 
All experiments are performed on a Windows system running on a PC equipped with an Intel Core i7-8550U CPU at 1.80 GHz, 4 cores, and 16 GB of RAM.

\subsection{Value of Learning}
\label{sec: VoL}
In this section, we examine the impact of learning on facility location decisions and objective values under various parameter settings. 
To this end, we first evaluate the {\it value of learning} (VoL) as the trade-off between the improvement in the objective value and the cost of collecting samples, as defined in \eqref{eq: VoL}. %We also investigate how the value of learning changes under different levels of demand uncertainty and learning efficiency. 
We then consider settings in which learning is available only to a subset of customers and compare the resulting value of learning with that in the full-learning case.
% All experiments in this section were conducted on the instance p41, where $|\Cs{I}|=10$ and $|\Cs{J}|=90$.
All experiments in this section were conducted on an adapted version of instance p41, with $|\Cs{I}|=10$, $|\Cs{J}|=20$, and $p=10$, i.e., the cardinality constraint is redundant. We also set $\rho=1$ and $\eta=34$.

\subsubsection{Full Learning}
In this section, we assume that learning is possible for all customers $j \in \Cs{J}$, a setting referred to as {\it full learning}. We investigate the VoL for  
%For all experiments in this section, we set the sampling cost to $d_j = 1$ for all customers and use $\delta = 10$ when defining the penalty cost. To study the effect of uncertainty and learning efficiency, we consider 
$\alpha \in \{1,\ldots,5\}$, 
 $\omega \in \{0.4,1,5\}$, and $d_j=d \in [0,2]$ for $j \in \mathcal{J}$, with increments of $0.1$. 
% \textcolor{red}{Why does Figure 1.a show $\alpha$ up to 5, while here it says $\le 3$?}
VoL is defined as %Recall that $\alpha$ controls the initial demand uncertainty through $\sigma_{j}^2 = \alpha \mu_j$, while $\omega$ controls the rate at which learning reduces uncertainty.
%For each parameter combination, we solve the model with and without learning and compute the VoL, defined as
% \begin{equation}
% \label{eq: VoL}
% \mathrm{VoL}
% =
%  \frac{\textrm{obj}(1) - \textrm{obj}(3)}{
% \textrm{obj}(1)
% }
% \times 100\%.    
% \end{equation}
% Here, $\textrm{obj}(1)$ denotes the optimal objective value for the model without learning, i.e., \eqref{eq: FL_DI_first}, where $\xi_j \sim \mathrm{N}(\mu_j,\sigma_j^2)$ for all $j \in \mathcal{J}$. On the other hand, $\textrm{obj}(3)$ denotes the optimal objective value for the model with learning, i.e., \eqref{eq: FL_DI_first_learning}. 
\begin{equation}
\label{eq: VoL}
\mathrm{VoL}
=
\frac{
F\big(x^\ast(0),0\big)-F\big(x^\ast(n^\ast),n^\ast\big)
}{
F\big(x^\ast(0),0\big)
}
\times 100\%.    
\end{equation}
Here, $x^\ast(0)$ denotes the optimal facility location decisions for the model without learning, i.e., \eqref{eq: FL_DI_first}, where $\xi_j \sim \mathrm{N}(\mu_j,\sigma_j^2)$ for all $j \in \mathcal{J}$. Similarly, $x^\ast(n^\ast)$ denotes the optimal facility location decisions for the model with learning, i.e., \eqref{eq: FL_DI_first_learning}, corresponding to the optimal solution $(x^\ast,n^\ast)$. Accordingly, $F\big(x^\ast(0),0\big)$ and $F\big(x^\ast(n^\ast),n^\ast\big)$ represent the optimal objective values without and with learning, respectively, where $F$ is defined in \eqref{eq: abstract_model}.
As defined, VoL quantifies the relative change in the total cost induced by changes in location decisions and by the reduction in uncertainty due to sampling.

Figure~\ref{fig:vol_sensitivity} reports the VoL.  %for the instances adapted from ``p41" in \citet{holmberg1999exact}. 
%shows the effect of the initial uncertainty level $\alpha$ under different values of the learning efficiency parameter $\omega$.
%The results show that 
Figure~\ref{fig:vol_alpha} shows that, for a fixed value of $\omega$, the VoL increases with $\alpha$. This is expected because greater initial demand uncertainty, given by $\sigma_j=\sqrt{\alpha\mu_j}$, increases the potential benefit of reducing uncertainty through learning. Moreover, for a fixed value of $\alpha$, the VoL also increases with $\omega$. This is because a larger $\omega$ implies that the same sampling effort yields a greater reduction in uncertainty, making learning more effective.
%In this experiment, the sampling costs are set equal across customers, i.e., $d_j=d$ for all $j$, so $d$ denotes a vector of identical customer-level sampling costs.
Figure~\ref{fig:vol_d} shows that the VoL decreases as $d$ increases. This is because a higher sampling cost reduces the net benefit of acquiring demand information through learning.

% For a fixed level of uncertainty, the value of learning increases with $\omega$. Since $\omega$ controls the rate at which uncertainty is reduced, larger values of $\omega$ make each observation more informative and therefore increase the benefit of learning.

% \begin{longtable}{cccc@{\hspace{1.2em}}cc@{\hspace{1.2em}}cc}

% \caption{Optimal objective values and value of learning for different uncertainty levels.}
% \label{tab:vol_results} \\

% \toprule
% Instance & $(I,J)$ & $\rho$ & $\eta$ 
% & $\alpha$ & $\omega$ 
% & Optimal Objective & VoL (\%) \\
% \midrule
% \endfirsthead

% \caption[]{(continued)} \\
% \toprule
% Instance & $(I,J)$ & $\rho$ & $\eta$ 
% & $\alpha$ & $\omega$ 
% & Optimal Objective & VoL (\%) \\
% \midrule
% \endhead

% \midrule
% \multicolumn{8}{r}{(continued on next page)} \\
% \endfoot

% \bottomrule
% \endlastfoot

% % ---------- p41 ----------
% \multirow{9}{*}{p41} 
% & \multirow{9}{*}{(10,90)} 
% & \multirow{9}{*}{0.5} 
% & \multirow{9}{*}{3}
% & 1 & 0.1 & 82183.22 & 1.21 \\
% & & & & 1 & 1   & 81324.85 & 2.24 \\
% & & & & 1 & 5   & 81146.71 & 2.45 \\
% & & & & 2 & 0.1 & 82720.17 & 3.02 \\
% & & & & 2 & 1   & 81435.91 & 4.52 \\
% & & & & 2 & 5   & 81171.93 & 4.83 \\
% & & & & 3 & 0.1 & 83102.07 & 4.86 \\
% & & & & 3 & 1   & 81514.86 & 6.67 \\
% & & & & 3 & 5   & 81193.88 & 7.04 \\

% \end{longtable}

\begin{figure}[!tb]
    \centering

    \subfloat[VoL vs. uncertainty parameters \(\alpha\) and \(\omega\), with $d=1$.]{
        \includegraphics[width=0.48\textwidth]{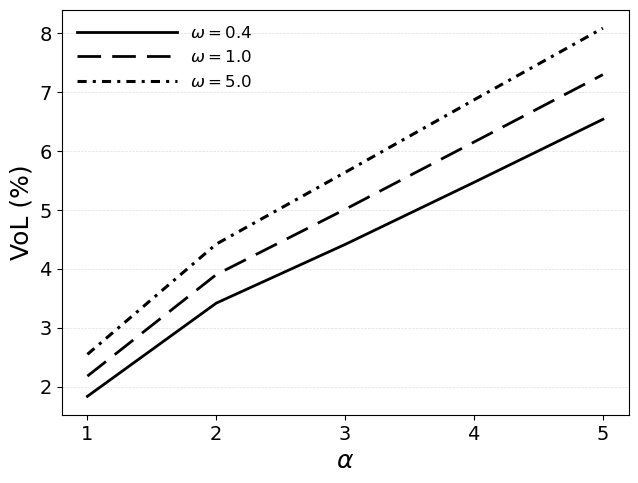}
        \label{fig:vol_alpha}
    }
    \hfill
    \subfloat[VoL vs. sampling cost \(d\), with $\alpha=3$ and $\omega=5$.]{
        \includegraphics[width=0.48\textwidth]{ 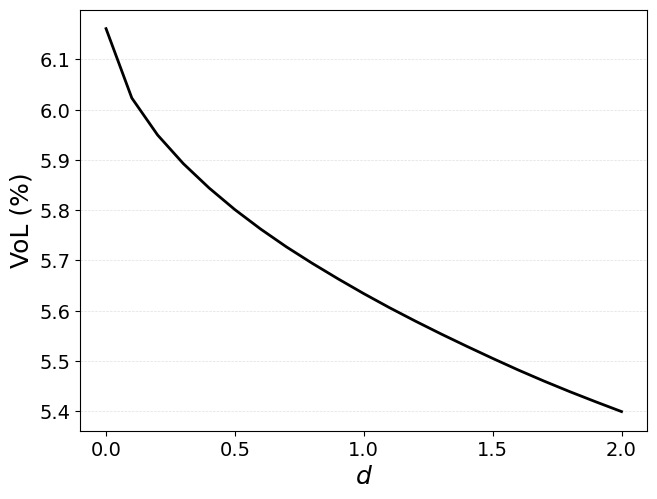}
        \label{fig:vol_d}
    }

    \caption{Sensitivity of the value of learning to uncertainty parameters and sampling cost.}
    \label{fig:vol_sensitivity}
\end{figure}

\subsubsection{Partial Learning}
\label{sec:partial_learning}
In practice, limitations in time or data availability may restrict demand learning to only a subset of customers. We refer to this setting as {\it partial learning}. Let $1 \le L \le |\Cs{J}|$ denote the number of customers whose demands can be learned, while the remaining customers keep their initial demand uncertainty $\sigma_j$. To account for the effect of customer selection, we consider two approaches for each value of $L$: (i) selecting ten random learning sets and (ii) selecting the $L$ customers with the lowest mean demand $\mu_j$, referred to as a {\it demand-ordered} learning set. %For both approaches, we compute the VoL using \eqref{eq: VoL}. %For the random learning sets, we also compute a {\it location-based} VoL in each of the ten replications for every value of $L$ by comparing $F\big(x^\ast(0),n^\ast\big)$ with $F\big(x^\ast(n^*),n^\ast\big)$, and report the mean across the replications. Since both objective values are evaluated under the same learned distribution $n^\ast$, their relative difference captures the cost improvement due to the change in the location decision as a result of the suboptimality of $\big(x^\ast(0),n^\ast\big)$.

Figure~\ref{fig:partial_learning_boxplot} reports the VoL for $L \in \{1,5,10,15,20\}$. The boxplots summarize the VoL over the ten random learning sets and the solid line represents the VoL under the demand-ordered approach. %, and the dashed line shows the mean location-based VoL. 
Overall, the VoL increases with the number of learned customers. For the random learning sets, this increase is approximately linear. This behavior is consistent with the separable structure of the expected recourse cost and the assumption that customer demands are independent. Under these assumptions, each additional learned customer contributes incrementally, causing the total VoL to grow approximately in proportion to $L$. In contrast, the demand-ordered curve is convex because larger values of $L$ add customers with higher demand means and, since $\sigma_j=\sqrt{\alpha\mu_j}$, greater initial uncertainty. These customers generally provide larger cost reductions from learning. However, their contribution also depends on transportation and penalty costs, so demand alone does not determine the VoL.

\begin{figure}[!tb]
\centering
\includegraphics[width=0.5\textwidth]{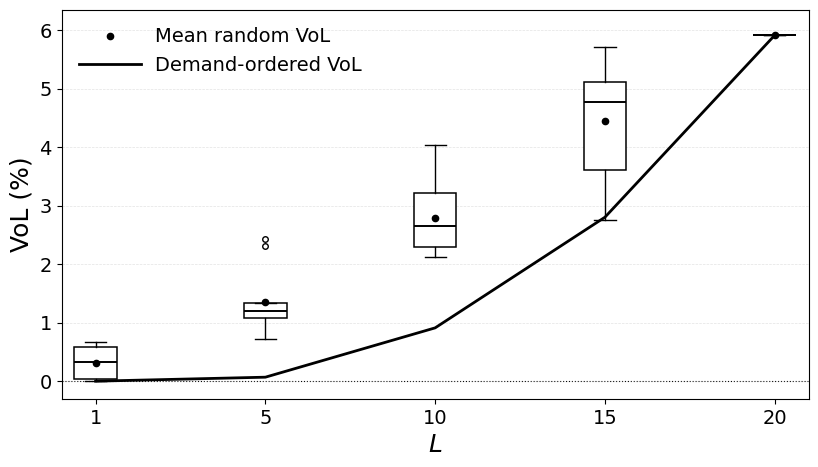}
\caption{Value of learning as a function of the number of learned customers, $L$. Each boxplot shows the VoL over ten random learning sets. The solid line represents the VoL under demand-ordered learning.}
\label{fig:partial_learning_boxplot}
\end{figure}

% \begin{table}[h]
% \centering
% \small
% \begin{tabular}{c c c c | c c c c}
% \toprule
% Instance & $(m,n)$ & $\rho$ & $\eta$ & $\alpha$ & $\omega$ & $L$ & VoL (\%) \\
% \midrule

% \multirow{10}{*}{p41} & \multirow{10}{*}{(10,90)} & \multirow{10}{*}{0.5} & \multirow{10}{*}{3}
% & \multirow{10}{*}{3} & \multirow{10}{*}{2}
% & 1  & $0.08 \pm 0.05$ \\
% & & & & & & 10 & $0.81 \pm 0.20$ \\
% & & & & & & 20 & $1.61 \pm 0.27$ \\
% & & & & & & 30 & $2.33 \pm 0.07$ \\
% & & & & & & 40 & $3.24 \pm 0.16$ \\
% & & & & & & 50 & $3.94 \pm 0.19$ \\
% & & & & & & 60 & $5.08 \pm 0.22$ \\
% & & & & & & 70 & $5.63 \pm 0.28$ \\
% & & & & & & 80 & $6.43 \pm 0.09$ \\
% & & & & & & 90 & 7.17 \\

% \bottomrule
% \end{tabular}
% \caption{Value of learning (VoL) under partial learning for instance p41. Each value is reported as mean $\pm$ standard deviation over five random selections of customers.}
% \label{tab:partial_learning}
% \end{table}

\subsection{Comparative Statics of Learning Decisions and Objective Value}
\label{sec: cs}

In this section, we illustrate the theoretical results established in Section~\ref{sec: comparativestatics} by examining the sensitivity of learning decision and objective value to key learning parameters, $d_j$, $\sigma_j=\sqrt{\alpha \mu_j}$, and $\omega$, for a single customer.
All experiments in this section were conducted on instance p41, with $|\Cs{I}|=10$, $|\Cs{J}|=90$, and $p=\left\lfloor |\Cs{I}|/2 \right\rfloor$. We also set $\rho=0.5$, $a_i^D=1$, and $u_i=0$ for all $i\in\Cs{I}$.  

%For the experiments in Sections~6.3 and~6.4, we set $a_i^D=1$ and $u_i=0$ for all $i\in\Cs{I}$, and impose the cardinality constraint $\sum_{i\in\Cs{I}}x_i=\left\lfloor |\Cs{I}|/2 \right\rfloor$.
 %particularly Theorems~\ref{thm:cs_min} and~\ref{thm:cs_value}. Since the learning subproblem is separable across customers, we examine the behavior of a single customer.

Figure~\ref{fig:comparative_statics} summarizes the comparative statics results on the optimal learning decision \(n_j^\ast\) and the corresponding optimal objective value \(V_j^\ast\). Figures~\ref{fig:cs_n_d} and~\ref{fig:cs_V_d} depict the sensitivity of results as the sampling cost \(d_j\) varies. Consistent with Theorem~\ref{thm:cs_min}, the optimal learning effort decreases as sampling becomes more expensive. The objective value increases with \(d_j\), but at a decreasing rate. This concave pattern reflects the trade-off between the marginal cost and marginal benefit of learning: as \(d_j\) increases, \(n_j^\ast\) decreases, which limits the additional cost imposed by higher sampling prices. Once \(d_j\) becomes sufficiently large, learning is no longer cost-effective, \(n_j^\ast=0\), and the objective value becomes flat because the solution coincides with that of the no-learning case. Figures~\ref{fig:cs_n_sigma} and~\ref{fig:cs_V_sigma} show the corresponding results as the initial demand uncertainty \(\sigma_{j}\) varies. Consistent with Theorems~\ref{thm:cs_min} and \ref{thm:cs_value}, both the optimal learning effort and objective value increase, respectively, as \(\sigma_{j}\) gets larger. %increases with \, which is also consistent with Theorem~\ref{thm:cs_min}. The optimal objective value increases as well, because higher initial uncertainty increases the expected cost associated with uncertain demand. 
Figures~\ref{fig:cs_n_omega} and~\ref{fig:cs_V_omega} show the sensitivity of \(n_j^\ast\) and \(V_j^\ast\) to the learning rate \(\omega\) (see Section~\ref{sec: design}). %, which is introduced in this section. 
%The theoretical argument is analogous to the proof in Theorem~\ref{thm:cs_min} and is omitted for brevity. 
Observe that both \(n_j^\ast\) and \(V_j^\ast\) decrease as \(\omega\) increases. A larger value of \(\omega\) reduces demand uncertainty, i.e., $\frac{\partial h_j}{\partial \omega}<0$, while making each additional unit of learning effort more effective at reducing uncertainty, i.e., $\frac{\partial^2 h_j}{\partial n_j \partial \omega}>0$. Consequently, less sampling effort is needed, leading to a lower optimal objective value. Furthermore, the effect of increasing $\omega$ on demand uncertainty is opposite to that of increasing $\sigma_j$ (see Assumption~\ref{assum:properties}), and accordingly, the comparative statics of the optimal learning decision and objective value are also reversed.

\begin{figure}[!tb]
    \centering

    % ---------- Row 1: optimal learning decision ----------
    \subfloat[$n_j^\ast$ vs. sampling cost $d_j$.]{
        \includegraphics[width=0.315\textwidth]{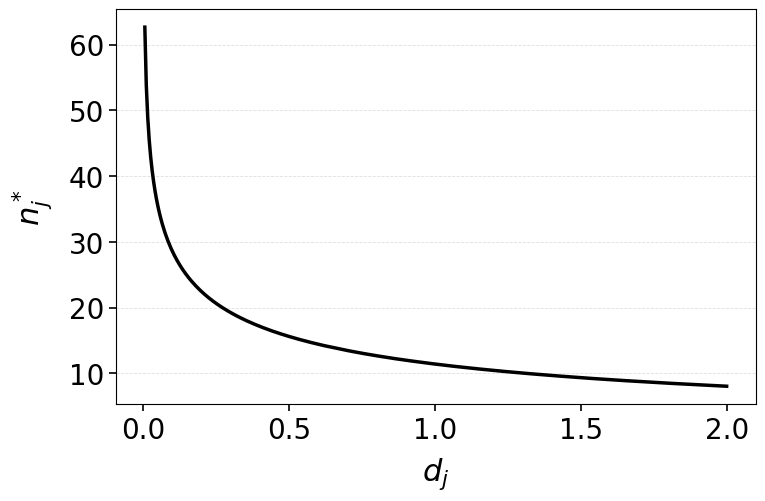}
        \label{fig:cs_n_d}
    }
    \hfill
    \subfloat[$n_j^\ast$ vs. initial demand uncertainty $\sigma_{j}$.]{
        \includegraphics[width=0.315\textwidth]{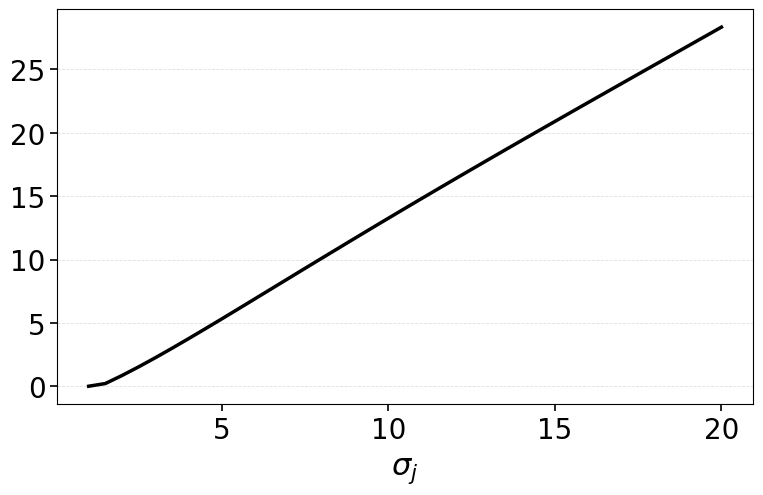}
        \label{fig:cs_n_sigma}
    }
    \hfill
    \subfloat[$n_j^\ast$ vs. learning rate $\omega$.]{
        \includegraphics[width=0.315\textwidth]{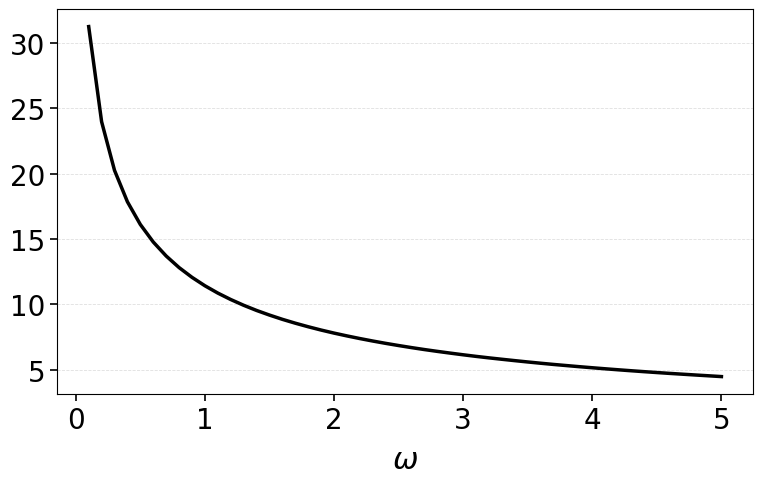}
        \label{fig:cs_n_omega}
    }

%    \vspace{1.0em}

    % ---------- Row 2: optimal objective value ----------
    \subfloat[$V_j^\ast$ vs. sampling cost $d_j$.]{
        \includegraphics[width=0.315\textwidth]{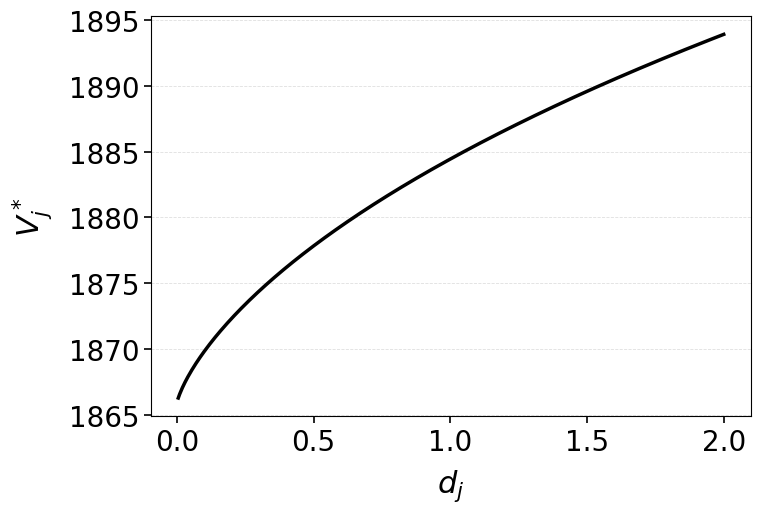}
        \label{fig:cs_V_d}
    }
    \hfill
    \subfloat[$V_j^\ast$ vs. initial demand uncertainty $\sigma_{j}$.]{
        \includegraphics[width=0.315\textwidth]{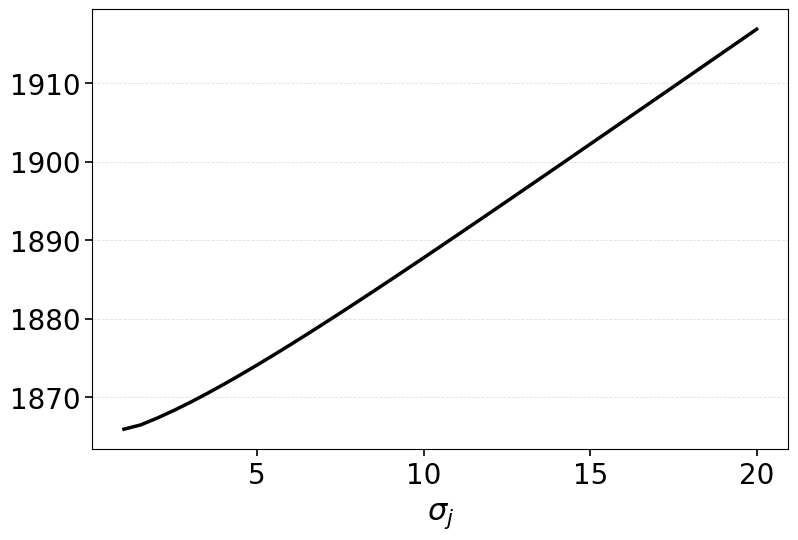}
        \label{fig:cs_V_sigma}
    }
    \hfill
    \subfloat[$V_j^\ast$ vs. learning rate $\omega$.]{
        \includegraphics[width=0.315\textwidth]{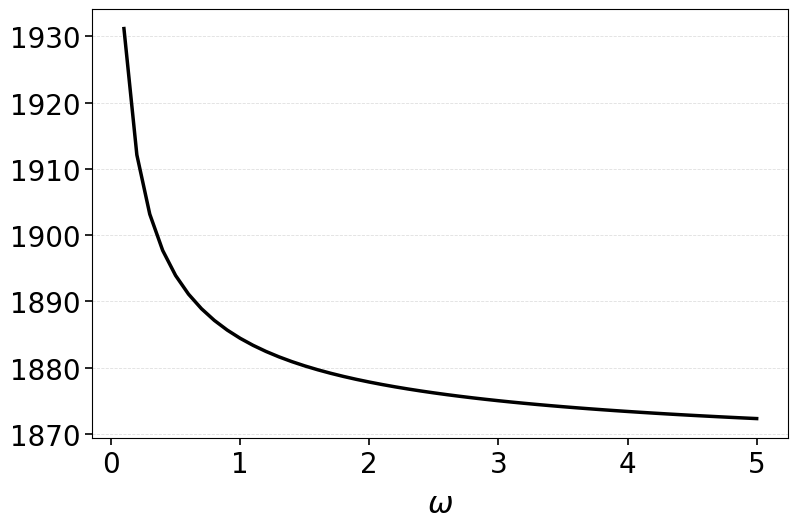}
        \label{fig:cs_V_omega}
    }

%    \vspace{1.0em}
    
    \caption{Comparative statics of the learning decisions and objective value for a fixed customer.}
    \label{fig:comparative_statics}
\end{figure}

% \begin{figure}[htbp]
%     \centering

%     % ---------- Row 1: sampling cost ----------
%     \subfloat[$n_j^\ast$ vs. sampling cost $d_j$.]{
%         \includegraphics[width=0.485\textwidth]{ n vs d.png}
%         \label{fig:cs_n_d}
%     }
%     \hfill
%     \subfloat[$V_j^\ast$ vs. sampling cost $d_j$.]{
%         \includegraphics[width=0.485\textwidth]{ V vs d.png}
%         \label{fig:cs_V_d}
%     }

%     \vspace{1.5em}

%     % ---------- Row 2: initial uncertainty ----------
%     \subfloat[$n_j^\ast$ vs. initial uncertainty $\sigma_{j}$.]{
%         \includegraphics[width=0.485\textwidth]{ n vs s0.png}
%         \label{fig:cs_n_sigma}
%     }
%     \hfill
%     \subfloat[$V_j^\ast$ vs. initial uncertainty $\sigma_{j}$.]{
%         \includegraphics[width=0.485\textwidth]{ V vs s0.png}
%         \label{fig:cs_V_sigma}
%     }

%     \vspace{1.5em}

%     % ---------- Row 3: learning efficiency ----------
%     \subfloat[$n_j^\ast$ vs. learning efficiency $\omega$.]{
%         \includegraphics[width=0.485\textwidth]{ n vs w.png}
%         \label{fig:cs_n_omega}
%     }
%     \hfill
%     \subfloat[$V_j^\ast$ vs. learning efficiency $\omega$.]{
%         \includegraphics[width=0.485\textwidth]{ V vs w.png}
%         \label{fig:cs_V_omega}
%     }

%     \caption{Comparative statics of the optimal learning decision and the optimal objective value for a fixed customer.}
%     \label{fig:comparative_statics}
% \end{figure}

\subsection{Computational Performance of the Proposed Algorithms}
\label{sec: comp}

Now, we present a comparative computational study of several solution approaches to solve \eqref{eq: abstract_model}. These include the proximal block coordinate descent solution approaches proposed in Sections~\ref{sec: alg_SAA} and \ref{sec: alg_PWL}, denoted as \texttt{SAA} and \texttt{PWLA}, respectively. 
As a benchmark, we implement an \texttt{Enumeration} approach that exhaustively evaluates all feasible binary facility location decisions \(x\); there are \(\binom{|\Cs{I}|}{p}\) feasible solutions. For each fixed solution \(x\), the resulting problem in the learning decisions \(n\) decomposes into $|\Cs{J}|$ univariate convex optimization problems, each of which can be solved independently. % using first-order derivative information. %Since this approach searches over all feasible facility location decisions, it provides a natural benchmark whenever enumeration is computationally tractable.
For comparison, we use Knitro, an off-the-shelf nonconvex solver, to solve \eqref{eq: abstract_model} directly over the joint decisions $(x,n)$, which is denoted by \texttt{Knitro}. We compare the computational performance of all four approaches in terms of computational time and solution quality.

\subsubsection{Experimental Setup and Implementation Details}

We present the computational results for instances adapted from ``p41", ``p51", and ``p56" \citep{holmberg1999exact}, which differ in the number of facilities and customers. For \texttt{SAA} and \texttt{PWLA}, we varied the approximation parameter $M$, corresponding to the number of samples and breakpoints, respectively, with $M \in \{50,100,200\}$ for \texttt{SAA} and $M \in \{5,10,20\}$ for \texttt{PWLA}. 
We set $\rho=0.5$ and $\eta=3$ for instances p41 and p51, and $\rho=0.7$ and $\eta=4$ for instance p56. Moreover, $a_i^D=1$ and $u_i=0$ for all $i\in\Cs{I}$, whereas   
$p=\left\lfloor |\Cs{I}|/2 \right\rfloor$.
For each experiment, we set a time limit of 21600 seconds. 

%To provide a structured comparison, we vary the parameters that determine the
%problem size and the approximation quality. These include the number of
%facilities and customers, the number of breakpoints used in the PWL
%approximation, and the number of samples used in the SAA. The
%remaining model parameters are kept fixed in this set of experiments so that the
%comparison focuses on the performance of the solution approaches rather than on
%changes in the underlying instance.

For \texttt{SAA}, the samples are generated using Latin Hypercube Sampling (LHS). Compared with standard Monte Carlo sampling, LHS provides more uniform coverage of the underlying sample space, and in our preliminary experiments, resulted in lower sampling error and shorter computational times. To account for sampling variability, each  \texttt{SAA} experiment is repeated 
over 10 independently generated sample sets. %, and the reported results are aggregated across these runs.
Furthermore, we implement \texttt{SAA} without a proximal term in the $n$-step. Our preliminary experiments indicate that incorporating the proximal term slows convergence while yielding the same local minimizer. We emphasize that the proximal term is included only to guarantee convergence by ensuring a sufficient decrease in the objective value; see Theorem~\ref{thm:min_convergence_SAA}. In practice, however, omitting the proximal term resulted in faster convergence without affecting the final solution in our experiments.
%Our preliminary experiments revealed that this strategy improves the approximation accuracy and typically requires fewer breakpoints than a uniform discretization of the entire interval.

For \texttt{PWLA}, we employ nonuniform breakpoints to improve approximation accuracy,  allocating approximately 80\% of the breakpoints to the interval \([-3,3]\), where the standard normal cdf and pdf exhibit the greatest variation. 
%a larger portion of them to the region where the standard normal cdf and pdf vary most rapidly. Specifically, approximately 80\% of the breakpoints were placed within the interval $[-3,3]$, while 
The remaining breakpoints are distributed over the tails. 
The breakpoints are generated globally and shared across all pairs $(j,k) \in  \Cs{J} \times \Cs{I}$ to construct the approximation function $S_M$, defined in \eqref{eq: pwl}. To determine the global breakpoint range, we first estimate the range of standardized terms $\frac{g_{kj(x)}}{h_j(n_j;\sigma_j)}=\frac{\hat{\xi}_{\pi(k)j}(x) - \mu_j}{h_j(n_j;\sigma_j)}$, over all feasible facility location decisions $x$ and learning decisions $n$. This is accomplished by computing the range of $\hat{\xi}_{\pi(k)j}(x) - \mu_j$ over facility location decisions $x$ and combining it with the range of $h_j(n_j;\sigma_j)$. Once this global interval is determined, the nonuniform breakpoints are constructed over the entire range. During optimization, each pair $(k,j)$ is used only the subset of breakpoints within its feasible interval, while breakpoints corresponding to points outside the interval are fixed to zero.

Although Knitro can approximate first-order derivatives using finite differences, this approach requires repeated objective function evaluations and exhibit poor computational performance in our preliminary experiments. 
%finite difference evaluations require repeated objective evaluations and lead to slow performance in our preliminary experiments. 
Therefore, for \texttt{Knitro}, we supply the analytical first-order derivatives with respect to both \(x\) and \(n\) to improve the solver performance. 

\subsubsection{Computational Results}
%Consistent with the notation in Section~\ref{sec: algorithms}, we use \(M\) to denote the approximation size for the approximation-based methods. In the \texttt{PWLA} approach, \(M\) denotes the number of breakpoints, while in the \texttt{SAA} approach, \(M\) denotes the number of samples.
For each solution approach $\texttt{alg} \in \{\texttt{enumeration}, \linebreak  \texttt{Knitro}, \texttt{PWLA}, 
\texttt{SAA}\}$, we evaluate the resulting solution $(x^\ast,n^\ast)$ using
the true objective function, $F$, defined in \eqref{eq: abstract_model}. We denote the corresponding objective value by $\textrm{obj}(\texttt{alg})$. %Here, the true objective function refers to the objective function in \eqref{eq: FL_DI_first_learning}, where the expectation term is evaluated using the closed-form expression in Proposition~\ref{prop: recourse_j}. 
We then compute a relative {\it gap} with respect to a 
{\it reference} objective function value, $\textrm{obj}^0$, as 
\begin{equation*}
    \mathrm{Gap}(\texttt{alg}): =
    \frac{\textrm{obj}(\texttt{alg}) - \textrm{obj}^0}
    {\textrm{obj}^0} \times 100  \%. 
\end{equation*}
%This evaluated value is then used to compute
%the reported gap. Thus, the gap is based on true objective evaluations rather than the approximate objective values obtained from the \texttt{PWLA} or \texttt{SAA} formulations. Let \(V^{\mathrm{method}}\) denote the true objective value obtained from a given method, and let \(V^{\mathrm{ref}}\) denote the reference objective value. The percentage gap is computed as
The reference objective value is obtained from \texttt{Enumeration} whenever it is computationally tractable; otherwise, we use the solution from \texttt{PWLA} 
with a sufficiently large number of breakpoints.
This choice is supported by the results for instance p41, where \texttt{PWLA} produced the same solution as \texttt{Enumeration}.
%This choice is supported by the ``p41", where the resulting solution from \texttt{PWLA} matches that of \texttt{enumeration}.

%For each SAA replication, we first solve the SAA formulation and then evaluate
%the resulting solution using the true objective function. The reported runtime
%and gap for SAA are given 
Table~\ref{tab:pwl_saa_comparison} summarizes the computational times and relative gaps of
the four approaches, For \texttt{SAA}, the reported numbers are  
as the mean \(\pm\) standard deviation over 10
replications.
The results demonstrate that \texttt{PWLA} provides the best balance between solution quality and computational efficiency. For instance p41, \texttt{Enumeration} is computationally tractable within the time limit and serves as an exact benchmark. Both \texttt{Knitro} and \texttt{PWLA} achieve zero gap, whereas \texttt{PWLA} requires substantially less computational time. For instance p51, \texttt{Knitro} also achieves a zero gap but requires more than 1600 seconds, whereas \texttt{PWLA} reaches the same solution quality in less than 30 seconds. For instance p56, \texttt{Knitro} does not close the gap within the six-hour time limit, while \texttt{PWLA} continues to obtain zero-gap solutions in short computational times. \texttt{SAA} also produces high-quality solutions across all instances;  however, its computational time increases substantially as the sample size, $M$, increases. Overall, these results indicate that \texttt{PWLA} scales more effectively than the alternative approaches while maintaining high solution quality.

\begin{table}[!htb]
\small
\centering
\caption{Comparison of \texttt{enumeration}, \texttt{Knitro}, \texttt{PWLA}, and \texttt{SAA}  across different instances.}
\label{tab:pwl_saa_comparison}
 \begin{threeparttable}
\begin{tabular}{c c c c c c}
\toprule
Instance & $(|\Cs{I}|, |\Cs{J}|)$ & Method & $M$ & Runtime (s) & Gap (\%) \\
\midrule

\multirow{8}{*}{p41}
& \multirow{8}{*}{$(10,90)$}
& \texttt{Enumeration} & -- & 509.09 & 0 \\ 
\cmidrule(lr){3-6}

&  & \texttt{Knitro} & -- & 402.19 & 0 \\
\cmidrule(lr){3-6}

&  & \multirow{3}{*}{\texttt{PWLA}}
& 5  & 3.27 & 0 \\
&  &  & 10 & 4.35 & 0 \\
&  &  & 20 & 15.28 & 0 \\
\cmidrule(lr){3-6}

&  & \multirow{3}{*}{\texttt{SAA}}
& 50  & $18.91 \pm 1.23$  & $0.009 \pm 0.011$ \\
&  &  & 100 & $71.47 \pm 7.42$  & $0.002 \pm 0.003$ \\
&  &  & 200 & $210.95 \pm 7.06$ & 0 \\

\midrule

\multirow{7}{*}{p51}
& \multirow{7}{*}{$(20,100)$}
& \texttt{Knitro} & -- & 1606.02 & 0 \\
\cmidrule(lr){3-6}

&  & \multirow{3}{*}{\texttt{PWLA}}
& 5  & 16.05 & 0 \\
&  &  & 10 & 23.69 & 0 \\
&  &  & 20 & 27.88 & 0 \\
\cmidrule(lr){3-6}

&  & \multirow{3}{*}{\texttt{SAA}}
& 50  & $186.09 \pm 20.74$   & $0.008 \pm 0.006$ \\
&  &  & 100 & $587.14 \pm 101.34$  & $0.003 \pm 0.004$ \\
&  &  & 200 & $201.35 \pm 6.50$  & 0 \\

\midrule

\multirow{7}{*}{p56}
& \multirow{7}{*}{$(30,200)$}
& \texttt{Knitro} & -- & $>21600$ & $3.43^{\dagger}$ \\
\cmidrule(lr){3-6}

&  & \multirow{3}{*}{\texttt{PWLA}}
& 5  & 88.42 & 0 \\
&  &  & 10 & 59.84 & 0 \\
&  &  & 20 & 137.88 & 0 \\
\cmidrule(lr){3-6}

&  & \multirow{3}{*}{\texttt{SAA}}
& 50  & $607.32 \pm 54.80$ & $0.064 \pm 0.040$ \\
&  &  & 100 & $2295.20 \pm 180.13$  & $0.014 \pm 0.009$ \\
&  &  & 200 & $6058.21 \pm 262.41$ & $0.004 \pm 0.003$ \\

\bottomrule
%\multicolumn{6}{l}{\footnotesize $^{\dagger}$Knitro was allowed to run for 6 hours; the reported gap is based on the best feasible solution found within that time.}
\end{tabular}

\begin{tablenotes}
\footnotesize

\item[$^{\dagger}$] The reported gap is based on the best feasible solution found within the time-limit.
\end{tablenotes}

 \end{threeparttable}
\end{table}

\section{Conclusion and Future Research}
\label{sec: conclusion}
%%%%%%%%%%%%%%%%%%%%%%%%%%%%%%%%%%%%%%%%%%%%%%%%%

%SUMMARY

This paper studies a two-stage stochastic capacitated facility location problem in which the decision maker can invest in learning the demand before it is realized. In the proposed model, facility location and learning decisions are made jointly in the first stage, where learning reduces demand uncertainty. This setting captures a practical planning trade-off: information can improve future operating performance, but it must be acquired at a cost.
We derived a closed-form expression for the expected recourse cost under independent normal demands, which leads to a nonlinear/nonconvex mixed-integer program with separability in the continuous learning decisions. 
Building on this structure, we established local comparative statics showing how optimal learning effort and the corresponding local optimal value respond to changes in sampling costs and initial demand uncertainty. 
These results formalize the intuition that learning becomes less attractive as it becomes more expensive and more beneficial as demand uncertainty increases.
We also developed proximal block coordinate descent algorithms that alternate between updating facility location and learning decisions. To handle the nonlinearity of the expected recourse function, we propose sample average and piecewise-linear approximations.
% two approximation schemes were proposed: a sample average approximation and a piecewise-linear approximation.
Theoretical results establish the convergence of the solution algorithms and consistency of the set of stationary points and local minimizers. 
%CONCLUSION/INSIGHT
The numerical results show that learning can substantially reduce expected costs, especially when initial demand uncertainty is high, learning is effective in reducing uncertainty, and sampling costs are low. They also show that, although location and learning decisions are made jointly in the first stage, learning can change the optimal location decision by changing the demand distribution, providing an additional source of cost improvement beyond uncertainty reduction. Partial learning experiments further show that even when learning is available only for a subset of customers, it can still provide meaningful value. In the computational comparison, the piecewise-linear approximation provides the strongest balance between solution quality and runtime, and scales more effectively than alternative methods, including enumeration and a direct solve by a commercial nonconvex solver.

% FUTURE RESEARCH
Several extensions are natural. First, the current model assumes independent customer demands and learning decisions that reduce variances while leaving means unchanged under a normal demand model. Future work could consider correlated demand, learning mechanisms that update both means and variances, or other probabilistic models. Second, robust or distributionally robust variants could be developed for settings in which the demand model itself may be misspecified.

%\bibliographystyle{informs2014}
%\bibliography{ref}

\ECSwitch %%% COMMENTED OUT
%\ECDisclaimer
%%%%%%%%%%%%%%%%%%%%%%%%%%%%%%%%%%%%%%%%%%%%%%%%%%%%%%%%%%
%%% Main head for the e-companion
\ECHead{Electronic Companion}

\section{Closed-Form Expression of the Objective Function}
\label{secEC:recourse}

%In this section, %~\ref{secEC:closed-form}, we present and prove results that help derive a closed-form expression of the expected recourse function, as characterized in Proposition~\ref{prop: recourse_j}. 

%We then establish the structure of the resulting closed-form objective function, as defined in \eqref{eq: abstract_model}, in Section~\ref{secEC:closed-form-structure}.

%\subsection{Closed-Form Expression}
%\label{secEC:closed-form}

Recall that given a fixed \(x\) and \(\xi\), \(R(x, \xi)\), as defined in~\eqref{eq: FL_DI_second}, is decomposable across customers. 
Let \(R_j(x, \xi_j)\) denote the cost associated with customer $j \in \Cs{J}$, with $R(x, \xi) = \sum_{j \in \Cs{J}} R_j(x, \xi_j)$. Consequently, 
$\ep{\xi}{R(x, \xi)} = \sum_{j \in \Cs{J}} \ep{\xi_j}{R_j(x, \xi_j)}$.  
Hence, throughout this section, we derive closed-form expressions for a single customer $j \in \Cs{J}$. Suppressing the index $j$ for notational simplicity, we have 
\begin{equation}
    \label{eq: FL_DI_second_j}
        R(x, \xi) = \min_{y,s} \Big\{ \sum_{i \in \Cs{I}}  c_{i} y_{i} + c_{0} s : \sum_{i \in \Cs{I}} y_{i} + s = \xi, \ y_{i} \leq D_{i} x_i, \ i \in \Cs{I}, \ 
        y_{i} \geq 0, \ i \in \Cs{I},\;
        s \geq 0 \Big\},
\end{equation}
where $\Cs{I}$ is a relabeling of facilities such that $c_0 > c_1 \geq c_2 \geq \dots \geq c_{|\Cs{I}|} \geq 0$.

For a fixed $x$, we first obtain a closed-form expression for the recourse function $R(x,\xi)$, as defined in~\eqref{eq: FL_DI_second_j}, and then combine it with the distribution of $\xi$, to characterize $\ep{\xi}{R(x,\xi)}$. 
Lemma~\ref{lem: affine_cost_j} characterizes \(R(x, \cdot)\) as a maximum of a finite number of affine functions. 
\begin{lemma}
\label{lem: affine_cost_j}
    Consider a fixed $x$ and $\xi$. Then, we have 
    %\begin{equation*}
    %   \label{eq: affine}
    %    \begin{aligned}
            $R(x, \xi) = \max_{k=0}^{|\Cs{I}|} \ \{a_{k}(x) + c_{k} \xi\}$,
    %    \end{aligned}
    %\end{equation*}
    where $a_{|\Cs{I}|}(x)=0$ and 
    %\begin{equation}
    %\label{eq: intercept_affine}
    %    \begin{aligned}
            $a_{k}(x) := \sum_{i=k+1}^{|\Cs{I}|} D_{i} (c_{i} - c_{k}) x_{i}$, $k = 0, \ldots, |\Cs{I}|-1$. 
    %    \end{aligned}
    %\end{equation}
    
\end{lemma}

\proof{Proof.}

Let $\beta$ and $\alpha_i$, $i \in \Cs{I}$, be the dual variables corresponding to the constraints in \eqref{eq: FL_DI_second_j}. The dual problem can be written as 
    %\begin{equation*}
        %\label{eq: subdual}
        %\begin{aligned}
            $\max_{\alpha,\beta} \Big\{ \beta \xi + \sum_{i \in \Cs{I}} \alpha_i D_i x_i: \beta + \alpha_i \leq c_i, \; i \in \Cs{I}, \; 
            \beta \leq c_0, \; 
            \alpha_i \leq 0, \; i \in \Cs{I} \Big\}$,
        %\end{aligned}
    %\end{equation*}
    where extreme points of its feasible region can be identified through two cases:
    \begin{itemize}
        \item Case 1: \(\beta = c_0\). Then, for all \(i \in \Cs{I}\), we have either \(\alpha_i = 0\) or \(\alpha_i = c_i - c_0\). 
        Note that \(\alpha_i \leq c_i - c_0 < 0\) since \(c_0 > c_i\), making \(\alpha_i = 0\) redundant. 
        Hence, \(\beta = c_0\) and \(\alpha_i = c_i - c_0 \) for $i \in  \Cs{I}$ form an extreme point with objective value \(a_0(x) + c_0 \xi\).

        \item Case 2: \(\beta < c_0\). Then, for all \(i  \in \Cs{I}\), we have either \(\alpha_i = 0\) or \(\alpha_i = c_i - \beta\). 
        Moreover, there must exist some \(k\) such that \(\alpha_k = c_k - \beta = 0\) so that at least \(|\Cs{I}| + 1\) constraints are binding at an extreme point. 
        For this extreme point, we have \(\beta = c_k ,\) and for \(i \neq k,\) \(\alpha_i = \min(0, c_i - c_k ),\) 
        yielding the objective value \(a_k(x) + c_k \xi.\)
    \end{itemize}
    The maximum objective value corresponding to these extreme points yields $R(x,\xi)$.
\hfill \Halmos
\endproof

Proposition~\ref{prop: piecewise_cost_j} characterizes $R(x, \cdot)$ as a convex increasing piece-wise linear function. 
\begin{proposition} \label{prop: piecewise_cost_j}
    Consider a fixed $x$. Let $\hat{\xi}_{k}(x):=\sum_{i=k}^{|\Cs{I}|} D_i x_i$, $k=1,\ldots,|\Cs{I}|$, with the conventions $\hat{\xi}_{0}(x)=+\infty$ and $\hat{\xi}_{|\Cs{I}|+1}(x)=-\infty$. Then, $R(x,\xi)=a_k(x)+c_k\xi$ for $\hat{\xi}_{k+1}(x)\le \xi < \hat{\xi}_{k}(x)$, $k=0,\ldots,|\Cs{I}|$, 
    where $a_k(x)$ is defined in Lemma~\ref{lem: affine_cost_j}. Moreover, $R(x,\cdot)$ is a convex increasing piece-wise linear function in $\xi$. 
\end{proposition}

\proof{Proof.}
For notational simplicity, let us suppress \(x\). By Lemma~\ref{lem: affine_cost_j}, we have \(R(\xi) = \max_{k=0}^{I} \bar{h}_k(\xi) \), where \(\bar{h}_k(\xi) = a_k + c_k \, \xi\),  $k=0, \dots, |\Cs{I}|$; hence, it is a convex piece-wise linear function.  
Fix \(k=0, \dots, |\Cs{I}|-2\). Functions \(\bar{h}_k(\xi)\) and \(\bar{h}_{k+1}(\xi)\) intersect at \(\hat{\xi}_{k+1} = \sum_{i = k+1}^{|\Cs{I}|} D_i x_i\), and we have \(\bar{h}_k(\hat{\xi}_{k+1}) = \bar{h}_{k+1}(\hat{\xi}_{k+1}) = \sum_{i=k+1}^{|\Cs{I}|} D_i c_i x_i\). Therefore, $\max \{\bar{h}_k(\xi), \bar{h}_{k+1}(\xi)\}$ is equal to $\bar{h}_{k+1}(\xi)$ if $\xi < \hat{\xi}_{k+1}$ and $\bar{h}_k(\xi)$, otherwise. 
    % \begin{equation}
    %     \label{eq: maxh(xi)1}
    %     \begin{aligned}
    %         \max \{\bar{h}_k(\xi), \bar{h}_{k+1}(\xi)\} =
    %         \begin{cases}
    %             \bar{h}_{k+1}(\xi), &  \quad \xi < \hat{\xi}_{k+1},\\
    %             \bar{h}_k(\xi), & \quad \xi \geq \hat{\xi}_{k+1}.
    %         \end{cases}
    %     \end{aligned}
    % \end{equation}
    Observe that \(a_{k+2} \leq \bar{h}_{k+1}(\hat{\xi}_{k+1})\) as 
    \[a_{k+2} - \bar{h}_{k+1}(\hat{\xi}_{k+1}) = \sum_{i=k+3}^{I} D_i(c_i-c_{k+2})x_i - \sum_{i=k+1}^{|\Cs{I}|} D_i c_i x_i = -c_{k+2} \sum_{i=k+3}^{|\Cs{I}|} D_i x_i - \sum_{i=k+1}^{k+2} D_i c_i x_i \leq 0.\] 
    Given this and the facts that \(c_0 > c_k \geq c_{k+1} \geq 0\) and \(a_k \leq a_{k+1} \leq 0\), we conclude that $\max \{\bar{h}_{k+1}(\xi), \bar{h}_{k+2}(\xi)\}$ is equal to $\bar{h}_{k+2}(\xi)$ if $\xi < \hat{\xi}_{k+2}$ and $\bar{h}_{k+1}(\xi)$, otherwise. 
    % \begin{equation}
    %     \label{eq: maxh(xi)2}
    %     \begin{aligned}
    %          \max \{\bar{h}_{k+1}(\xi), \bar{h}_{k+2}(\xi)\} =
    %          \begin{cases}
    %             \bar{h}_{k+2}(\xi), & \quad \xi < \hat{\xi}_{k+2},\\
    %             \bar{h}_{k+1}(\xi), & \quad \xi \geq \hat{\xi}_{k+2}.
    %         \end{cases}
    %     \end{aligned}
    % \end{equation}
    Observe that \(0 \leq \hat{\xi}_{k+2} \leq \hat{\xi}_{k+1}\) and \(0 \leq \bar{h}_{k+2}(\hat{\xi}_{k+2}) \leq \bar{h}_{k+1}(\hat{\xi}_{k+1})\). These two facts, combined with the previous two conclusions, imply that $\max \{\bar{h}_k(\xi), \bar{h}_{k+1}(\xi), \bar{h}_{k+2}(\xi)\}$ is equal to $\bar{h}_{k+2}(\xi)$ if $\xi < \hat{\xi}_{k+2}$, $\bar{h}_{k+1}(\xi)$ if $\hat{\xi}_{k+2} \leq \xi < \hat{\xi}_{k+1}$, and $\bar{h}_k(\xi)$, otherwise. 
    % \begin{equation*}
    %     \begin{aligned}
    %         \max \{\bar{h}_k(\xi), \bar{h}_{k+1}(\xi), \bar{h}_{k+2}(\xi)\} =
    %         \begin{cases}
    %             \bar{h}_{k+2}(\xi), & \quad \xi < \hat{\xi}_{k+2},\\
    %             \bar{h}_{k+1}(\xi), & \quad \hat{\xi}_{k+2} \leq \xi < \hat{\xi}_{k+1},\\
    %             \bar{h}_k(\xi), & \quad \xi \geq \hat{\xi}_{k+1}.
    %         \end{cases}
    %     \end{aligned}
    % \end{equation*}
    Repeating this process for \(k=0, \dots, |\Cs{I}|-2\) yields \(R(x, \xi)\).
    Also, as  \(c_0 > c_1 \geq \ldots \geq c_{|\Cs{I}|} \geq 0\), \(a_1 \leq a_{2} \leq \ldots \leq a_{|\Cs{I}|} \leq 0\), we conclude $R(x, \cdot)$ is increasing. 
\hfill \Halmos
\endproof

Before proving Proposition~\ref{prop: recourse_j}, we present a lemma related to a normal distribution, whose proof is immediate by noting that $\phi'(z) = -z \, \phi(z)$. 
\begin{lemma} \label{lem:2}
    Suppose that \(\zeta\) follows a standard normal distribution, i.e., \(\zeta \sim N(0, 1)\). Then, $\int_{a}^{b} \zeta \, \phi(\zeta) \, d\zeta = \phi(a) - \phi(b)$.
\end{lemma}

%\proof{Proof.}
%This is immediate by noting that $\phi'(z) = -z \, \phi(z)$. 
%\hfill \Halmos 
%\endproof

% \begin{lemma} \label{lem:4}
%     Suppose that \(\zeta \sim N(\mu, \sigma^2)\) with pdf \(f(\cdot)\). Then,
%     \begin{equation*}
%         \begin{aligned}
%             \int_{\bar{a}}^{\bar{b}} (a + c\,\zeta) f(\zeta)  d\zeta
%             = (a + c\,\mu)  \Phi\left(\frac{\bar{b} - \mu}{\sigma}\right) 
%             - c\sigma\,\phi\left(\frac{\bar{b} - \mu}{\sigma}\right) - (a + c\,\mu) \Phi\left(\frac{\bar{a} - \mu}{\sigma}\right) 
%             + c\sigma\,\phi\left(\frac{\bar{a} - \mu}{\sigma}\right).
%         \end{aligned}
%     \end{equation*}
% \end{lemma}

%\proof{Proof.}
%A change of variable and an application of Lemma~\ref{lem:2}  yield the result. 
%\hfill \Halmos 
%\endproof

\proof{Proof of Proposition~\ref{prop: recourse_j}.}

For notational simplicity, let us suppress \(x\) and suppose that $\xi \sim N(\mu,\sigma^2)$. By Lemma~\ref{lem: affine_cost_j}, we have \(R(\xi) = \max_{k=0}^{|\Cs{I}|} \bar{h}_k(\xi) \), where \(\bar{h}_k(\xi) = a_k + c_k \, \xi\),  $k=0, \dots, |\Cs{I}|$. 
Let \(\hat{\xi}_0 := +\infty\) and \(\hat{\xi}_{|\Cs{I}|+1} := -\infty\). By Proposition~\ref{prop: recourse_j}, we have $\mathbb{E}[R(\xi)] = \sum_{k=0}^{|\Cs{I}|} \bar{H}_k$, where 
    $\bar{H}_k := \int_{\hat{\xi}_{k+1}}^{\hat{\xi}_k} h_k(\xi)\,f(\xi)\,d\xi$, $k = 0, \dots, |\Cs{I}|$, with \(f(\cdot)\) denoting the pdf of \(\xi\). For simplicity, define
    $\hat{z}_k := \frac{\hat{\xi}_k - \mu}{\sigma}$, $k = 0, \dots, |\Cs{I}|+1$.
    A change of variable and an application of Lemma~\ref{lem:2}  yield 
    %\begin{equation*}
        $\bar{H}_k = \left(a_k + c_k \, \mu \right) \, \Phi\!\left(\hat{z}_k\right) - c_k \, \sigma \, \phi\!\left(\hat{z}_k\right)- \left(a_k + c_k \, \mu \right) \, \Phi\!\left(\hat{z}_{k+1}\right) + c_k \, \sigma \, \phi\!\left(\hat{z}_{k+1}\right)$,   
    %\end{equation*}
    for $k = 0, \dots, |\Cs{I}|$. 
    Therefore,
    \begin{align*}
        \sum_{k=0}^{|\Cs{I}|} \bar{H}_k = &   (a_0 + c_0 \mu) \Phi\left(\hat{z}_0\right) - c_0 \sigma \phi\left(\hat{z}_0\right)
        + \sum_{k=1}^{|\Cs{I}|} (c_{k-1} - c_k)\Big[\left(\hat{\xi}_k - 
        \mu\right)  \Phi\left(\hat{z}_k\right)+\sigma \phi\left(\hat{z}_k\right) \Big] \\
        & {} \quad - (a_{|\Cs{I}|}+ c_{|\Cs{I}|} \mu) \Phi\left(\hat{z}_{|\Cs{I}|+1}\right) + c_{|\Cs{I}|} \sigma  \phi\left(\hat{z}_{|\Cs{I}|+1}\right),
    \end{align*}
    where we used the fact that $a_{k} - a_{k-1}= 
    \sum_{i=k+1}^{|\Cs{I}|} D_{i}(c_{i} - c_{k}) \, x_{i} - \sum_{i=k}^{|\Cs{I}|} D_{i}(c_{i} - c_{k-1}) \, x_{i}= \sum_{i=k+1}^{|\Cs{I}|} D_{i}(c_{k-1} - c_{k}) \, x_{i} - D_k(c_{k} - c_{k-1})  x_k = (c_{k-1} - c_{k}) \hat{\xi}_{k+1} + (c_{k-1} - c_{k}) D_k x_k= (c_{k-1} - c_{k}) \hat{\xi}_k
    $ for $k=1, \ldots, |\Cs{I}|$.
    Note that \(\hat{\xi}_0 = +\infty\), implying \(\Phi\!\left(\hat{z}_0\right) = 1\) and \(\phi\!\left(\hat{z}_0\right) = 0\). 
    Also, \(\hat{\xi}_{I+1} = -\infty\), implying \(\Phi\!\left(\hat{z}_{I+1}\right) = 0\) and \(\phi\!\left(\hat{z}_{I+1}\right) = 0\).
    These, combined with the definition of $\Psi(z)=z\Phi(z)+ \phi(z)$, and replacing $\sigma$ with $h(n;\sigma)$ complete the proof.
\hfill \Halmos
\endproof

\section{Convergence Proof of Algorithms}
\subsection{Preliminary Definitions from Variational Analysis}
\label{sec: EC_prelim}
For $y^\ast \in \Bs{R}^n$ and radius $\epsilon$, the $\epsilon$-neighborhood of $y^\ast$ is $\Bs{N}_{\epsilon}(y^\ast):=\{y \in \Bs{R}^n: \|y-y^\ast \| < \epsilon\}$. Let $T: \Bs{R}^n \rightrightarrows \Bs{R}^m$ be a point-to-set mapping. 
Normal cone of set $\Cs{S} \subseteq \Bs{R}^n$ at point $y$ is defined as $N_{\Cs{S}}(y):=\{s: s^\top(x-y) \le 0 \; \forall x \in S\}$.
For a closed set $\Cs{S} \subseteq \Bs{R}^n$, its indicator function is defined as $i_{\Cs{S}}(y)=0$ if $y \in \Cs{S}$ and $+\infty$, otherwise.
Function $u\mapsto Z(u)$ is $L$-Lipschitz if there exists $L>0$ such that $|Z(u) - Z(u^{\prime})| \leq L\|u-u^{\prime}\| \ \forall u,u^{\prime}$. % and for some $p\ge0$. %A family of random functions $\{Z(\cdot, \xi): \xi \in \Xi\}$ is equicontinuous on a set $\Cs{U}$ almost surely if for every $\varepsilon>0$, there exists a $\eta>0$ such that $|Z(u,\xi)-Z(u^{\prime},\xi)|<\varepsilon$ whenever $\|u-u'\| < \delta$, for all $u, u^{\prime} \in \Cs{U}$, and almost every $\xi \in \Xi$. 
The graph of $T$ is defined by $\Graph T 
:= \{ (x,y) \in \Bs{R}^n \times \Bs{R}^m : y \in T(x)\}$ and its domain is given by $\dom T:= \{ x \in \Bs{R}^n : T(x) 
\neq \emptyset\}$.
Similarly, the graph of a extended-real-valued function $f: \Bs{R}^n \to \Bs{R} \cup \{+\infty\}$ is defined by $\Graph f 
:= \{ (x,y) \in \Bs{R}^n \times \Bs{R}^m : y = f(x)\}$ and its domain is given by $\dom f:= \{ x \in \Bs{R}^n : f(x) < +\infty\}$.
For each $x \in \dom f$, the Fréchet subdifferential of $f$ at $x$, written $\widehat{\partial} f(x)$, is the set of vectors 
$v \in \mathbb{R}^n$ which satisfy
%\begin{equation*}   
$\liminf_{\substack{y \to x \\ y 
\neq x}}
\frac{f(y) - f(x) - \langle v, y - x \rangle}
{\|x - y\|}
\ge 0$.
%\end{equation*}
When $x 
\notin \dom f$, we set $\widehat{\partial} f(x) = \emptyset$.
The subdifferential of $f$, written $\partial f (x)$, is defined as follows
\begin{equation*}
    \partial f(x):= \left\{ v \in \Bs{R}^n \; : \; \exists\, x^k \to x,\; f(x^k) \to f(x),\; v^k \in \widehat{\partial} f(x^k),\; v^k \to v \right\}.
\end{equation*}
It is straightforward to verify from the definition the following closedness property
of $\partial f$. 
Let $\{(x^k, v^k)\}_{k \in \Bs{N}}$ be a sequence in 
$\Bs{R}^n \times \Bs{R}^n$ such that $(x^k, v^k) \in \Graph \partial f$ for all $k \in \Bs{N}$. 
If $(x^k, v^k) \to (x,v)$ and $f(x^k) \to f(x)$, then $(x,v) \in \Graph \partial f$.
The distance of point $b$ to a set $\Cs{B}$ is defined as $\dist(b,\Cs{B})=\inf \{x \in \Cs{B}: \|b-x\|\}$. 
Consider two sets $\Cs{B}^{1}, \Cs{B}^{2} \subseteq \Bs{R}^{n}$. The Hausdorff distance between $\Cs{B}^{1}$ and $\Cs{B}^{2}$ is denoted by   $d_{H}(\Cs{B}^{1}, \Cs{B}^{2})$ and is defined as 
%\begin{equation*}
 $   d_{H}(\Cs{B}^{1}, \Cs{B}^{2}):=\max\{ \sup_{b^{2} \in \Cs{B}^{2}} \ \dist (b^{2},\Cs{B}^{1}), \sup_{b^{1} \in \Cs{B}^{1}}  \ \dist (b^1,\Cs{B}^{2})\}$.
\begin{definition}[Quadratic Growth]
    \label{def: quadratic_growth}
    A proper and lower semicontinuous function $f : \Bs{R}^n \to \Bs{R} \cup \{+\infty\}$ has a quadratic growth at a local minimum $\bar{x}$ if there exist constants $\varepsilon > 0$ and $c > 0$ such that for all $y \in \Bs{N}_{\varepsilon}(\bar{x})$, we have 
    %\begin{equation*}
    $    f(y) \ge f(\bar x) + \frac{c}{2} \|y-\bar{x}\|^2$. 
    %\end{equation*}
\end{definition}

% \begin{definition}[Strong Convexity]
% \label{def: strong_convexity}
% A function $f : \Bs{R}^n \to \Bs{R} \cup \{+\infty\}$ is $\mu$-strongly convex if for all $x,y$ and $v \in \partial f(x)$, we have 
% %\begin{equation*}
% $f(y) \ge f(x) + \langle v, y-x \rangle + \frac{\mu}{2}\|y-x\|^2$, 
% %\end{equation*}
% with $\mu > 0$.
% Equivalently, $f$ is $\mu$-strongly convex if for all $x,y$ and $v \in \partial f(x)$ for all  $x,y$ and $v\in\partial f(x)$, $w\in\partial f(y)$, we have 
% %\begin{equation*}
% $\langle v-w,\; x-y\rangle \ge \mu\|x-y\|^2$.
% %\end{equation*}
% \end{definition}

\begin{definition}{(Kurdyka-Łojasiewicz (KL) function 
\citep[Definition~2.4]{attouch2013convergence})}
\label{def:KL}
    A function $f : \Bs{R}^n \to \Bs{R} \cup \{+\infty\}$  is said to have the KL property at $x^\ast \in \dom \partial f$ if there exist 
$\eta \in (0,+\infty]$, a neighborhood $U$ of $x^\ast$, 
and a continuous concave function 
$\vartheta : [0,\eta) \to \Bs{R}_+$ such that:
%\begin{enumerate}
    (i) $\vartheta(0)=0$,
    (ii) $\vartheta$ is $C^1$ on $(0,\eta)$,
    (iii) for all $s \in (0,\eta)$, $\vartheta'(s) > 0$, and 
    (iv) for all 
    $x \in U \cap \{x : f(x^\ast) < f(x) < f(x^\ast)+\eta \}$, KL holds:
    \begin{equation}
    \label{eq: KL_inequality}
    \vartheta'\big(f(x)-f(x^\ast)\big) \cdot 
    \dist \big(0,\partial f(x)\big)
    \ge 1.
    \end{equation}
%\end{enumerate}
Proper lower semicontinuous functions which satisfy the KL inequality at each point of $\dom \partial f$ are called KL functions. 
\end{definition}
\subsection{Technical Lemmas}

\begin{definition}[Blockwise Local Minimum]
\label{def:local_min}
A point $(x^\ast, n^\ast) \in \Cs{X} \times \Cs{N}$ is a blockwise local minimum of $F$ if $x^\ast \in\argmin_{x\in\Cs{X}} F(x,n^\ast)
$ and there exists some $\epsilon >0$ such that $F(x^\ast, n^\ast) \le F(x^\ast, n)$ for all $n \in \Cs{N} \cap \Bs{N}_\epsilon(n^\ast)$. 
\end{definition}

\begin{lemma}[KL Function]
\label{lem:KL}
    Functions $F, F_M, S_M: (x,n) \mapsto \Bs{R}$, as defined in \eqref{eq: abstract_model}, \eqref{eq: abstract_SAA}, and \eqref{eq: pwl}, respectively, are a tame function (in a suitable o-minimal structure) and are  KL functions. % on $\Bs{R}^{|\Cs{I}|} \times \Bs{R}^{|\Cs{J}|}$. 
\end{lemma}
\proof{Proof.}
    The stated functions are proper lower semicontinuous. Moreover, they are tame functions in the sense of \citet{bolte2009tame}, connecting to semialgebraic functions and satisfying the KL property \citep{attouch2013convergence}; thus, they are KL functions by Definition ~\ref{def:KL}. 
\hfill \Halmos 
\endproof  

\begin{lemma}[Sufficient Decrease]
\label{lem:sifficient_decrease}
Suppose there exists $K_0$ such that $x^k=x^\ast$ for all $k\ge K_0$. 
Step 2 of Algorithm \ref{alg:bcd} guarantee a nonincreasing sequence of $\{Q(x^\ast,n^{k};M)\}$ with a sufficient decrease, proportionate to $\|n^{k+1} - n^k\|$, unless $n^{k+1}=n^k$, i.e.,  $Q(x^\ast,n^{k+1};M) \le Q(x^\ast,n^{k};M) - \frac{\lambda_k}{2} \|n^{k+1} - n^k\|^2$. 
\end{lemma}
\proof{Proof.}
    %For the $x$-update, by the global optimality, we have $P(x^{k+1}, n^{k};M) \le P(x^{k}, n^{k};M)$.
By suboptimality, $Q(x^\ast, n^{k+1};M) \le Q(x^\ast, n^{k};M) - \frac{\lambda_k}{2} \|n^{k+1} - n^k\|^2$, 
%Summing these yields 
%\begin{equation}
%\label{eq:sufficient_decrease}
% F(x^{k+1}, n^{k+1}) \le F(x^{k}, n^{k}) - \frac{\lambda_k}{2} \|n^{k+1} - n^k\|^2, 
%\end{equation}
proving the sequence of values is nonincreasing and has a sufficient decrease, as $\lambda_k >0$, unless $n^{k+1}=n^k$.
\hfill \Halmos
\endproof
\begin{lemma}[Convergence to Critical Point]
\label{lem:convergence}
Suppose there exists $K_0$ such that $x^k=x^\ast$ for all $k\ge K_0$.
The sequence $\{n^k\}$ of Algorithm~\ref{alg:bcd} converges to a critical point of $Q(x^\ast,\cdot;M)$ over $\Cs{N}$.  
\end{lemma}

Part of the proof of Lemma~\ref{lem:convergence} follows similarly to the proof of \cite[Lemma~2.6]{attouch2013convergence}, with the exception that the primitive assumptions in that result need to be proved. We make a proper citation to that proof to relegate.

\proof{Proof of Lemma~\ref{lem:convergence}.}
%     We first show $\{n^k\}$ converges to some point $\bar{n}$. By Lemma \ref{lem:sifficient_decrease}, we have sufficient decrease. Hence, from \eqref{eq:sufficient_decrease}, for any $K>0$, we have 
% \begin{equation*}
%     \sum_{k=0}^{K} \frac{\lambda_k}{2} \|n^{k+1} - n^k\|^2 \le F(x^{0}, n^{0}) - F(x^{K+1}, n^{K+1}). 
% \end{equation*}
% Given that $\Cs{X}$ and $\Cs{N}$ are bounded and $F$ is continuous on $\Bs{R}^n \times \Bs{R}$,  $F(x^{K+1}, n^{K+1})$ is bounded below. Hence, the series $\sum_{k=0}^{\infty} \frac{\lambda_k}{2} \|n^{k+1} - n^k\|^2 $ converges, implying that $\|n^{k+1} - n^k\| \to 0$ as $k \to \infty$. 
First, since $\Cs{N}$ is bounded, the sequence $\{n^k\}$ admits a convergent subsequence $n^{k_t}\to \bar n$.
Let $G(n) = Q(x^\ast,n;M)+  i_{\Cs{N}}(n)$.  
Since $\{G(n^k)\}_{k >K_0}$ is a nonincreasing sequence by Lemma \ref{lem:sifficient_decrease}, the limit $\ell:=\lim_{k\to\infty} G(n^k)$ exists.
Moreover, by continuity of $G$ and $n^{k_t}\to \bar n$, we have $G(n^{k_t})\to G(\bar n)$, and thus $\ell=G(\bar n)$. In particular, $G(n^k)\ge G(\bar n)$ for all $k\ge K_0$.
By Lemma~\ref{lem:KL}, $G(x^\ast,n)$ is a KL function and its summation with the indicator function $i_{\Cs{N}}(n)$ preserves KL; hence, $G$ is a KL function. Let $U$, $\eta$, and $\vartheta: [0,\eta) \mapsto \Bs{R}$ be the objects appearing in Definition \ref{def:KL}. 
Let $\delta,\rho>0$ be such that $\Bs{N}_\delta(\bar{n}) \subset U$ with $\rho \in (0,\delta)$, with the condition that $\eta < \frac{\lambda_K}{2}(\delta - \rho)^2$. 
By continuity of $G$ and $\vartheta$, there exists $K \ge K_0$ such that for all $k \ge K$, we have 
\begin{equation}
\label{continuity1}
    G(\bar{n}) \le G(n^k) < G(\bar{n}) + \eta, 
\end{equation}
and 
\begin{equation}
\label{continuity2}
    \|\bar{n} - n^{K}\| + 2 \sqrt{\frac{G(n^K)-G(\bar{n})}{C_K}} + \frac{D_{K-1}}{C_K} \vartheta \big(G(n^K)- G(\bar{n})\big) < \rho,
\end{equation}
where $C_k:=\lambda_k/2$ and $D_k:=\lambda_k$.
The key to our proof is the following for all $j \ge K+1$:
    \begin{equation}
     n^j \in \Bs{N}_{\rho}(\bar{n}), \label{claim1}   
    \end{equation}
    \begin{equation}
        \label{claim2}
        \sum_{i=K+1}^{j} \|n^{i+1} - n^i\| + \|n^{j+1} - n^j\| \le \|n^{K+1} - n^{K}\| + 
\frac{D_{j-1}}{C_j} \Big[
\vartheta\big(G(n^{K+1}) - G(\bar{n})\big) -
\vartheta\big(G(n^{j+1}) - G(\bar{n})\big)
\Big].
    \end{equation}
Before, we prove \eqref{claim1} and \ref{claim2}, we claim that for a fixed $k \ge K+1$, if  $n^k \in \Bs{N}_{\rho}(\bar{n})$, then 
\begin{equation}
\label{claim3}
     2\|n^{k+1} - n^k\| \le \|n^k - n^{k-1}\|+ \frac{D_{k-1}}{C_{k}}\vartheta\big(G(n^k) - G(\bar{n})\big) - \vartheta\big(G(n^{k+1}) - G(\bar{n})\big).
\end{equation}
Clearly, if $n^{k+1}=n^k$, this inequality holds trivially. Assume that $n^{k+1}
\neq n^k$, which implies that $G(n^k)>G(n^{k+1})$ by Lemma \ref{lem:sifficient_decrease}; thus, $G(n^k) > G(\bar{n})$. 
By the optimality of $n^{k}$ for the $n$-step, we have as $\lambda_k(n^{k-1} - n^{k}) \in 
\nabla_n Q(x^\ast, n^{k};M) + N_{\Cs{N}}(n^{k})$. That is, $\lambda_k(n^{k-1} - n^{k}) \in \partial G(n^{k})$.  Therefore, 
%\begin{equation*}
$\text{dist}(0, \partial G(n^{k})) \leq D_k \|n^{k} - n^{k-1}\|$. 
%\end{equation*}
Moreover, by the KL inequality, we have $\text{dist}(0, \partial G(n^{k}))>0$ as $G(n^k) > G(\bar{n})$. Therefore, $n^k
\neq n^{k-1}$. 
Combining with the KL inequality, we have
\begin{equation}
%\label{KL}
 \vartheta'\big(G(n^k) - G(\bar{n})\big) \geq \frac{1}{\text{dist}(0, \partial G(n^{k}))} \ge \frac{1}{D_k \|n^{k} - n^{k-1}\|}.
\end{equation} 
Moreover, by Lemma \ref{lem:sifficient_decrease}, we have 
%\begin{equation}
%\label{sufficient_decrease}
$ G(n^k) - G(n^{k+1}) \geq C_k \|n^{k+1} - n^k\|^2$. 
%\end{equation}
The concavity of $\vartheta$, then implies that  
\begin{equation*}
\vartheta\big(G(n^k) - G(\bar{n})\big) - \vartheta\big(G(n^{k+1}) - G(\bar{n})\big) \geq \vartheta'\big(G(n^k) - G(\bar{n})\big) (G(n^k) - G(n^{k+1})\big),    
\end{equation*}
which, combined with the previous two conclusions, yields 
\begin{equation*}
\frac{D_{k-1} }{C_k}\Big[\vartheta\big(G(n^k) - G(\bar{n})\big) - \vartheta\big(G(n^{k+1}) - G(\bar{n})\big)\Big] \geq \frac{\|n^{k+1} - n^k\|^2}{\|n^{k} - n^{k-1}\| }.
\end{equation*}
Now, multiplying the inequality by $\|n^{k} - n^{k-1}\|$, taking the square root on both sides and using the identity $2\sqrt{ab} \le a+b$ yield \eqref{claim3}. 
We also claim that 
\begin{equation}
\label{claim4}
    n^k \in \Bs{N}_{\rho}(\bar{n}) \; \Rightarrow \; n^{k+1} \in \Bs{N}_{\delta}(\bar{n}) \ \text{with} \ G(n^{k+1})\ge G(\bar n). 
\end{equation}
To see this, note that by Lemma \ref{lem:sifficient_decrease}, we have 
%\begin{equation*}
    $\|n^{k+1} - n^{k} \| \le \sqrt{\frac{G(n^k)-G(n^{k+1})}{D_k}} \le \sqrt{\frac{\eta}{D_k}}< \delta - \rho$. 
%\end{equation*}
Hence, by the triangle inequality we have $\|n^{k+1}-\bar{n}\| < \delta$, which expresses that $n^{k+1} \in \Bs{N}_{\delta}(\bar{n})$. 
Claims \eqref{claim1} and \eqref{claim2} can now be proved by induction on $j$, similar to that of \cite[Lemma~2.6]{attouch2013convergence}, by noting that \eqref{continuity1}, \eqref{continuity2}, \eqref{claim3}, and \eqref{claim4} correspond to Eqs. (3), (4), (9), and (5) in that proof, respectively, while \eqref{claim1} and \eqref{claim2} correspond to Eqs. (6) and (7), respectively, after a proper translation of indices. 
%Take $j=K+1$. By \eqref{continuity1} and \eqref{claim4}, we have $G(n^{K+1}) \ge G(\bar{n})$ and $n^{K+1} \in \Bs{N}_{\delta}(\bar{n})$. Therefore, 
% \begin{equation}
% \label{claim5}
% \|n^{K+1} - n^{K} \| \le \sqrt{\frac{G(n^K)-G(n^{K+1})}{D_K}} \le \sqrt{\frac{G(n^K)-G(\bar{n})}{D_K}}. 
% \end{equation}
% So, \eqref{continuity2} and the triangle inequality yield
% \begin{equation*}
%     \|\bar{n} - n^{K+1}\| \le \|\bar{n} - n^{K}\| + \| n^{K} - n^{K+1}\| \le \|\bar{n} - n^{K}\|  +  \sqrt{\frac{G(n^K)-G(\bar{n})}{D_K}} < \rho,
% \end{equation*}
% that is, $n^{K+1} \in \Bs{N}_{\rho}(\bar{n})$. 
% Now, \eqref{claim3} with $k=K+1$ shows that \eqref{claim2} holds with $j=K+1$. 
% Suppose that \eqref{claim1} and \eqref{claim2} hold for some $j\ge K+1$. Then, using the triangle inequality and
% \eqref{claim2}, we have 
% \begin{align*}
%     \|\bar{n} - n^{j+1}\| & \le
% \|\bar{n} - n^K\| +
% \|n^K - n^{K+1}\|
% + \sum_{i=K+1}^{j} \|n^{i} - n^{i+1}\| \\ 
% & \le
% \|\bar{n} - n^{K}\|
% +
% 2\|n^{K} - n^{K+1}\|
% +
% \frac{D_{j-1}}{C_j}
% \Big[
% \vartheta\big(G(n^{K+1}) - G(\bar{n})\big)
% -
% \vartheta\big(G(n^{j+1}) - G(\bar{n})\big)
% \Big]
% \end{align*}
% Now, combined with \eqref{continuity2} and \eqref{claim5}, we conclude that $n^{j+1} \in \Bs{N}_{\rho}(\bar{n})$. 
% Hence, \eqref{claim3} holds with $k=j+1$. 
% Combining that with \eqref{claim2} for $j$, yield \eqref{claim2} for $j+1$. 
A direct use of \eqref{claim2} yields 
%\begin{align*}
    $\sum_{k=K+1}^j \|n^{k+1} - n^k\| \leq 
    \|n^{K+1} - n^{K}\| +  2\vartheta\big(G(n^{K+1}) - G(\bar{n})\big)$, 
%\end{align*}
and, thus, $\sum_{k=K}^\infty \|n^{k+1} - n^k\|$ converges. Hence, the sequence $\{n^k\}$ is Cauchy and convergent to  $\tilde{n}=\bar{n}$.
Moreover, the optimality condition for the $n$-step implies that $\lambda_k(n^{k-1}-n^k)\in \partial G(n^k)$ for all $k\ge K_0+1$. Therefore,
%\begin{equation*}
$\dist\big(0,\partial G(n^k)\big) \le \lambda_k\|n^k-n^{k-1}\| \to 0$.
%\end{equation*}
In particular, there exists $v^k \in \partial G(n^k)$ such that $\|v^k\|\to 0$. By the closedness of $\Graph\partial G$ and $G(n^k)\to G(\bar n)$, we obtain $0\in \partial G(\bar n)$, i.e., $\bar n$ is a critical point of $Q(x^\ast,\cdot;M)$ over $\Cs{N}$.
\hfill \Halmos
% Now, by summing \eqref{claim3} from $k = K$ to $N$, for some $N$, we have 
% \begin{align*}
%     2 \sum_{k=K}^N \|n^{k+1} - n^k\| & \leq \sum_{k=K}^{N}  \|n^k - n^{k-1}\| + \sum_{k=K}^{N} \frac{D_{k-1}}{C_k}  \vartheta\big(G(n^K) - G(\bar{n})\big) - \vartheta\big(G(n^{N+1}) - G(\bar{n})\big),\\
%     &\le \sum_{k=K}^{N} \|n^k - n^{k-1}\| +   2 \vartheta\big(G(n^K) - G(\bar{n})\big),\\
%     & \le \sum_{k=K}^{N}  \|n^k - n^{k-1}\| +  2 \vartheta\big(G(n^K) - G(\bar{n})\big)
% \end{align*}
% where the second inequality is due to the nonnegativity of $\vartheta$ and the third inequality is due to the fact that $\frac{D_{k-1}}{C_k} \le 2$. 
\endproof

\subsection{Convergence Proof of Theorem \ref{thm:min_convergence}}
\label{secEC:min_convergence} 

Without a proximal term, Theorem~\ref{thm:min_convergence} can be proved using standard block coordinate descent convergence arguments; see, e.g., \cite{zangwill1969nonlinear} and \citet[Proposition3.7.1]{bertsekas2016nonlinear}. We provide a proof for the case in which a proximal term is incorporated into the $n$-step, thereby enabling a direct application when such a term is required for convergence; see Section~\ref{secEC: alg_SAA}.

\proof{Proof of Theorem \ref{thm:min_convergence}.}
Suppose that the sequence $\{n^k\}$ converges to some $n^\ast\in \Cs{N}$. By Assumption~\ref{assum:stability}, there exists $K_0$ such that $x^{k}=x^\ast$ for all $k\ge K_0$, where $x^\ast = T(n^k) \in \argmin_{x \in \Cs{X}} F(x,n^{k})$.
By Lemma~\ref{lem:convergence}, $n^\ast$ satisfies the stationarity condition
$0 \in \partial G(n^\ast)$, 
where $G(n):=F(x^\ast,n)+i_{\Cs{N}}(n)$. 
The descent property in Lemma \ref{lem:sifficient_decrease} ensures that $n^\ast$ cannot be a local maximizer.
We show that $(x^\ast,n^\ast)$ is a local minimum of $F$ over $\Cs{X}\times \Cs{N}$.
Since $\Cs{X}$ is finite, there exists $r>0$ such that $\Cs{X}\cap \Bs{N}_r(x^\ast)=\{x^\ast\}$.
Therefore, for every $(x,n)\in (\Cs{X}\times \Cs{N})\cap \big(\Bs{N}_r(x^\ast)\times \Cs{N}\big)$ we have $x=x^\ast$ and hence $F(x,n)=F(x^\ast,n)$.
We claim that that $n^\ast$ is an (isolated) strict local minimizer of $F(x^\ast,\cdot)$ over $\Cs{N}$.
For $j \in \Cs{J}$, either $n^\ast_j=0$ or $n^\ast_j \in (0,b)$ (Lemma \ref{lem:min-g}), where in a neighborhood of $n^\ast_j$, $f_j(x,n_j)$ satisfies either a local linear growth or a local quadratic growth, respectively (Corollary \ref{cor:min-g}). Therefore, there exists $\delta>0$ such that $G(n) > G(n^\ast)$ over $\Cs{N} \cap \Bs{N}_{\delta}(n^\ast)$. 
%Because $G$ satisfies a quadratic growth over $\Bs{N}_{\delta}(n^\ast)$, there exists $c>0$ such that $G(n) \ge G(n^\ast)+\frac{c}{2}\|y-n^\ast\|^2$. Take any $y\in \Bs{N}_{\delta}(n^\ast)$ with $y
%\neq n^\ast$. Then $\|y-n^\ast\|^2>0$; thus, $G(n)> G(n^\ast)$. 
That is, $n^\ast$ is a strict local minimum of $G$. Then, $n^\ast$ is an strong (isolated) local minimizer of $G$. 
Therefore, $(x^\ast,n^\ast)$ is a local minimum of $F$ over $\Cs{X}\times \Cs{N}$ 
%Combining with the finiteness of $\Cs{X}$ (so $x=x^\ast$ locally), we conclude that $(x^\ast,n^\ast)$ is a local minimum of $F$ over $\Cs{X}\times \Cs{N}$ 
in the sense of Definition~\ref{def:local_min}. 
\hfill \Halmos
% The proximal objective in the $y$-step,
% \begin{equation*}
% H(n)= F(x^{k+1}, y) + \frac{\lambda_k}{2} (y-n^k)^2 + i_{[0,b]}(n),    
% \end{equation*}
% is $(\lambda_k - L_{k+1})$-strongly convex on $[0,b]$. Thus, $y$-subproblem admits a unique minimizer $n^{k+1}$ and $H$ is strongly metrically subregular at $n^{k+1}$ by Lemma \ref{lem:strong_convexitn_metric_subregularity}. 
% Given that the function $F(x^\ast,\cdot)$ is twice continuously differentiable on $[0,b]$, $G$ is prox-regular on $[0,b]$ by Lemma \ref{lem:prox_regualar}. 
% Under the strong convexity $H$ and its strong metric subregularity at the critical point $n^\ast$, quadratic growth $H$ near $n^\ast$ follows \cite[Theorem~3.3]{drusvyatskiy2018error}. Then, by \cite[Theorem~13.27]{rockafellar1998variational}, 
\endproof

\subsection{Convergence Proofs of the SAA}
\label{secEC: alg_SAA}
Throughout this section, recall that $F$ and $F_M$ are defined in \eqref{eq: abstract_model} and \eqref{eq: abstract_SAA}, respectively.  Also, define $\ell(\zeta^t;x,n)=  \sum_{j \in \Cs{J}} \sum_{i \in \Cs{I}} q_{ij} \left(\zeta^t h_j(n_j) + g_{ij}(x)\right)_{+}$, where $f_M(x,n)=\frac{1}{M} \sum_{t=1}^{M} \ell(\zeta^t;x,n)$ and  $f(x,n)=\e{\ell(\zeta;x,n)}$, with $\zeta \sim \textrm{Normal}(0,1)$. Finally, we define $\ell_{ij}(\zeta^t;x,n_j)=q_{ij} \big(\zeta^t h_j(n_j) + g_{ij}(x)\big)$. 

\begin{lemma}
    \label{lem:SAA}
    We have $\Psi(z)=\e{(\zeta+z)_+}$, where $\zeta \sim \textrm{N}(0,1)$ and $\Psi(z)$, defined in \eqref{eq:Psi}. 
\end{lemma}

\proof{Proof.} We have 
    %\begin{align*}
        $\e{(\zeta+z)_+} = \int_{-\infty}^{\infty} (\zeta+z)_+ \phi(\zeta) \,d\zeta = \int_{-z}^{\infty} (\zeta+z)\,\phi(\zeta) \,d\zeta
        =  z\int_{-z}^{\infty} \phi(\zeta) \,d\zeta + \int_{-z}^{\infty} \zeta\phi(\zeta) \,d\zeta$. Now, noting that $\phi'(\zeta) = -\zeta \, \phi(\zeta)$ completes the proof. 
         %& =z \Phi(z) - \int_{-z}^{\infty} \phi'(\zeta)\,d\zeta =  z \Phi(z) + \phi(z). 
         \hfill \Halmos
    %\end{align*} 
    
\endproof   

\begin{lemma}
    \label{lem:SAA_phi}
    We have $\phi(z)=\e{\zeta \one\{\zeta+z >0\}}$, where $\zeta \sim \textrm{N}(0,1)$. 
\end{lemma}

\proof{Proof.} We have 
    %\begin{align*}
        $\e{\zeta \one\{\zeta+z >0\}}=\int_{-z}^{\infty} \zeta \,\phi(\zeta) \,d\zeta=\int_{-z}^{\infty} -\phi'(\zeta) \,d\zeta=\phi(z)$, noting that $\phi'(\zeta) = -\zeta \, \phi(\zeta)$ and symmetry of $\phi(\cdot)$.
         \hfill \Halmos
    %\end{align*} 
    
\endproof

\proof{Proof of Theorem \ref{thm:min_convergence_SAA}.}
Let $\Omega_0$ denote the probability-one event on which all sampled realizations $\{\zeta^t\}_{t=1}^M$ are well-defined and the sample is fixed. We prove the result pathwise on $\Omega_0$.
Fix any $\omega\in\Omega_0$. For this realization, $F_M(\cdot,\cdot;\omega)$ is a deterministic objective on $\Cs{X}\times\Cs{N}$, and Algorithm \ref{alg:bcd} with $P(\cdot;M)=Q(\cdot;M)=F_M$ is exactly the same proximal block-coordinate algorithm as in Theorem \ref{thm:min_convergence}, with $F$ replaced by $F_M$.
By assumption, for this $M$, the  sequence $\{n^k\}$ converges to $n_M^\ast$, $x_M^\ast$ is the stability point under Assumption~\ref{assum:stability}, and $F_M(x_M^\ast,\cdot)$ satisfies a local quadratic growth at $n_M^\ast$ on $\Cs{N}$. 
Therefore, all primitives of Theorem \ref{thm:min_convergence} (applied to the deterministic function $F_M$) hold. Hence, the generated sequence converges to $(x_M^\ast,n_M^\ast)$, and this limit is a local minimum of $F_M$ over $\Cs{X}\times\Cs{N}$.
As the argument holds for $\omega\in\Omega_0$ and $\Pr(\Omega_0)=1$,  the conclusion holds with probability 1.
\hfill \Halmos
\endproof

Keys to the proof of  Theorem~\ref{thm:saa_stationary_consistency} are the uniform convergence of $F_M(x,\cdot)$ to $F(x,\cdot)$ and consistency of subdifferentials, stated next. 

\begin{lemma}[Uniform Convergence]
\label{lem:uniform_convergence}
For every $x \in \Cs{X}$, we have 
%\begin{equation*}
$\max_{n\in\Cs{N}}|F_M(x,n)-F(x,n)|\to 0$ as  $M\to\infty$,  
%\end{equation*}
almost surely (a.s.). 
\end{lemma}

\proof{Proof.}
For simplicity, we suppress $x$ and $\sigma$ in our notation. 
It is easy to verify that there exists a continuous (random) $B$ such that $\max_{n\in \Cs{N}}|\ell(\zeta;n)| \le B(\zeta)$ a.s., with $\e{B}<\infty$. Hence, by the strong law of large numbers (SLLN), a.s. 
for every $n \in \Cs{N}$, we have pointwise convergence: $F_M(n) \rightarrow F(n)$ as $M\to\infty$. 
%\begin{equation}
%\label{eq:pointwise_convergence}
%F_M(n) \rightarrow F(n) \qquad \text{as } M\to\infty.
%\end{equation}
%Similarly, 
%\begin{equation}
%\label{eq:convergence_dominance}
%$\frac{1}{M}\sum_{t=1}^M B(\zeta^t) \rightarrow \e{B}$ as $M\to\infty$, a.s. 
%\end{equation}
%Also, $y\mapsto \ell(\xi;y)$ is continuous on the compact interval $\Cs{N}=[0,b]$ for a.e.\ $\xi$.
We now show a random Lipschitz bound for $\ell(\zeta;n)$. For every $\zeta$, $n\in \Cs{N}$, and $s \in \partial_{n_j} \big(\ell_{ij}(\zeta;n_j)\big)_{+}$ for $j \in \Cs{J}$, we have 
%\begin{equation*}
$| s | \le |\zeta| \big| h'_j(n_j) \big| \sum_{i \in \Cs{I}} q_{ij}$. Let $\alpha_j:=\max_{n_j \in [0,b]} |h'_j(n_j)| \times |\Cs{I}|\max_{i \in \Cs{I}} q_{ij}$, where $\alpha_j <\infty$ given that $h'_j(n_j)$ is continuous, and hence bounded, on $[0,b]$. Define $L(\zeta):=|\zeta| \|\alpha\|$. 
By the mean-value theorem,
%\begin{equation}
%\label{eq:rand_lipschitz_ell}
$|\ell(\zeta;n)-\ell(\zeta;n^\prime)|\le L(\zeta)\|n-n^\prime\|$ for all $n,n^\prime \in\Cs{N}$ a.s.
Also, $\e{|\zeta|}=\sqrt{\frac{2}{\pi}}$; hence, $\e{L}<\infty$ and 
%\end{equation}
%Moreover, 
by SLLN, a.s. 
we have $\frac{1}{M}\sum_{t=1}^M L(\zeta^t) \rightarrow \e{L}$ as $M\to\infty$. 
%\begin{equation}
%\label{eq:convergence_rand_lipschitz_ell}
%\frac{1}{M}\sum_{t=1}^M L(\zeta^t) \rightarrow \e{L} \qquad \text{as } M\to\infty.  
%\end{equation}
%Therefore, the family of functions $\{\ell(\xi;\cdot): \xi\}$ is equicontinuous almost surely. 
Fix $\varepsilon>0$. If $\e{L}=0$, then $L(\xi)=0$ a.s. and $\ell(\xi;\cdot)$ is a.s. constant on $\Cs{N}$, so the claim is immediate from the pointwise convergence.  %\eqref{eq:pointwise_convergence}. 
Hence, consider $\e{L}>0$ and choose $\delta =\frac{\varepsilon}{6\e{L}}$. By the compactness of $\Cs{N}$, there exists a collection of points $\{n^1,\dots,n^r\}\subset \Cs{N}$ such that $\Cs{N} \subset \bigcup_{k \in [r]} \Bs{N}_{\delta}(n^k)$.
For each fixed $k\in[r]$, %\eqref{eq:pointwise_convergence} 
the pointwise convergence, gives $F_M(n^k)\to F(n^k)$ a.s. Since $r<\infty$, there exists $M_0:=M_0(\varepsilon)$ such that for all $M\ge M_0$ and all $k\in[r]$, 
%\begin{equation*}
$|F_M(n^k)-F(n^k)| \le \varepsilon/2$.
%\end{equation*}
Moreover, %by %\eqref{eq:convergence_rand_lipschitz_ell}, 
there exists $M_1$ such that for all $M \ge M_1$, we have $\frac{1}{M}\sum_{t=1}^M L(\zeta^t) \le 2 \e{L}$. 
Now take any $n \in \Cs{N}$ and choose $k\in[r]$ such that $n\in \Bs{N}_\delta(n^k)$, i.e., $\|n-n^k\| < \delta$. Then, for all $M \ge \max\{M_0,M_1\}$,
we have 
\begin{align*}
\big|F_M(n)-F(n)\big|
& \le \big|F_M(n)-F_M(n^k)\big| + \big|F_M(n^k)-F(n^k)\big| + \big|F(n^k)-F(n)\big| \\
& \le %\Big(\frac{1}{M}\sum_{t=1}^M L(\xi^t)\Big)\|n - n^k\| + \frac{\varepsilon}{2} + \e{L} \|n - n^k\|
%= 
\frac{\varepsilon}{2}  + \|n - n^k\| \Big( \e{L} + \frac{1}{M}\sum_{t=1}^M L(\xi^t)\Big)\le \frac{\varepsilon}{2} + \delta 3 \e{L}= \varepsilon.  
\hfill \Halmos
\end{align*}
%Thus, $\max_{n\in\Cs{N}}|F_M(n)-F(n)|\to 0$ as $M \to \infty$ a.s.
\endproof

\begin{lemma}[Subdifferential Consistency]
\label{lem:saa_subdiff_consistency}
Consider $x \in \Cs{X}$. 
Let $\{(n_M, v_M)\}_{M \in \Bs{N}}$ be a sequence such that $v_M\in \partial_n\big(F_{M}(x,\cdot)+i_{\Cs{N}}\big)(n_M)$. If $n_M \to \bar n$ and $v_M \to \bar v$ as $M \to \infty$, then $\bar v \in \partial_n\big(F(x,\cdot)+i_{\Cs{N}}\big)(\bar n)$ a.s. 
\end{lemma}

\proof{Proof.}
For simplicity, we suppress $x$ and $\sigma$ in our notation. 
For $i \in \Cs{I}$ and $j \in \Cs{J}$, it is easy to verify that there exists a continuous (random) $B$ such that $\max_{n_j\in  [0,b]} |\ell_{ij}(\zeta;n_j)| \le B(\zeta)$ a.s., with $\e{B}<\infty$.  
For every $\zeta$, define $G_{ij}(\zeta;n_j):=q_{ij} \zeta h'_j(n_j) \one\{\zeta h_j(n_j) + g_{ij}>0\}$, which is bounded above over $n_j \in [0,b]$. Hence, by the uniform SLLN, we have $\max_{n_j \in [0,b]} \big| \frac{1}{M} \sum_{t=1}^M G_{ij}(\zeta^t;n_j) - \e{G_{ij}(\zeta;n_j)}\big| \to 0$ as $M \to \infty$.  We also have $\e{G_{ij}(\zeta;n_j)}=q_{ij}h'_j(n_j)\phi\Big(\frac{g_{ij}(x)}{h_j(n_j)}\Big)$ by an application of Lemma~\ref{lem:SAA_phi}, which combined with 
$\frac{\partial F_j}{\partial n_j}(n_j)= d_j + h_j'(n_j) \sum_{i \in \Cs{I}}  q_{ij} \phi \left(\frac{g_{ij}}{h_j(n_j)}\right)$, lead to $\max_{n_j \in [0,b]} \big| d_j + \frac{1}{M} \sum_{t=1}^M \sum_{i \in \Cs{I} } G_{ij}(\zeta^t;n_j) - \frac{\partial F_j}{\partial n_j}(n_j) \big| \to 0$ as $M \to \infty$.  

Given that $F_M(n)$ is Lipschitz (see the proof of Lemma~\ref{lem:uniform_convergence}), we have $v_M=w_M+\eta_M $, where $w_M \in \partial_n F_{M}(n_M)$ and $\eta_M \in N_{\Cs{N}}(n_M)$. 
For $j \in \Cs{J}$, $w_M$ differs from $d_j + \frac{1}{M} \sum_{t=1}^M \sum_{i \in \Cs{I} } G_{ij}(\zeta^t;n_{M,j})$ only through terms satisfying $\zeta^t h_j(n_{M,j})+g_{ij}=0$. 
Because $\zeta$ has a continuous distribution and $\Cs{I}$ is finite, the number of such terms is uniformly finite a.s., and their contribution is bounded by a constant multiped by $\frac{1}{M} \max_{t=1}^{M} |\zeta^t|$, which goes to $0$ as $M \to \infty$. Thus, $w_M \to \nabla_n F(\bar n)$ as $M \to \infty$. 
Since $v_M \to \bar v$ as $M \to \infty$, we also have $\eta_M \to \bar v-\nabla_nF(\bar n)$ as $M \to \infty$.
Using the closedness property of the limiting subdifferential of $i_{\Cs{N}}$  (cf.\ the closedness of $\Graph\,\ \partial_n i_{\Cs{N}}$, or $\Graph\,\ N_{\Cs{N}}$), we have $\bar v-\nabla_nF(\bar n)\in N_{\Cs{N}}(\bar n)$. 
Because $F$ is continuously differentiable, $\partial_n\big(F(\cdot)+i_{\Cs{N}}\big)(\bar n)
= \nabla_nF(\bar n)+N_{\Cs{N}}(\bar n)$. Therefore, $\bar v\in \partial_n\big(F(\cdot)+i_{\Cs{N}}\big)(\bar n)$. 
\hfill \Halmos 
\endproof

\proof{Proof of Theorem \ref{thm:saa_stationary_consistency}.}
By the finiteness of $\Cs{X}$, Lemma~\ref{lem:uniform_convergence} implies that 
%\begin{equation*}
$\max_{x\in\Cs{X}}\max_{n\in\Cs{N}}|F_M(x,n)-F(x,n)|\to 0$  as $M\to\infty$   
%\end{equation*}
with probability 1. 
Let $(x_M,n_M)\in \Cs{S}_M$. Since $\Cs{X}$ is finite and $\Cs{N}$ is compact, the sequence $\{(x_M,n_M)\}$ admits a subsequence (not relabeled) such that $x_M\equiv \bar x$ and $n_M\to \bar n\in\Cs{N}$.
By definition of $\Cs{S}_M$, we have $0\in \partial_n\big(F_M(\bar x,\cdot)+i_{\Cs{N}}\big)(n_M)$ and $\bar x =T_M(n_M) \in \argmin_{z\in\Cs{X}} F_M(z,n_M)$.
%Since $F_M(\bar x,\cdot)\to F(\bar x,\cdot)$ uniformly and $n_M\to \bar n$, we have $F_M(\bar x,n_M)\to F(\bar x,\bar n)$ by continuity of $F$.
Fix any $z\in\Cs{X}$ and any $\varepsilon>0$. By uniform convergence, there exists $M_0(\varepsilon)$ such that for all $M\ge M_0(\varepsilon)$, 
%and all $n\in\Cs{N}$,
%\begin{equation*}
%$|F_M(\bar x,n)-F(\bar x,n)|\le \varepsilon$ %and 
%\qquad\text{and}\qquad
%$|F_M(z,n)-F(z,n)|\le \varepsilon$.    
%\end{equation*}
%In particular, for all $M\ge M_0(\varepsilon)$, we have
%\begin{equation*}
$F(\bar x,n_M)\le F_M(\bar x,n_M)+\varepsilon \le F_M(z,n_M)+\varepsilon \le F(z,n_M)+2\varepsilon$.   
%\end{equation*}
Letting $M\to\infty$ and then $\varepsilon\downarrow0$, continuity of $F$ yields $F(\bar x,\bar n)\le F(z,\bar n)$ for all $z\in\Cs{X}$. Therefore, $\bar x \in \argmin_{z\in\Cs{X}} F(z,\bar n)$. %, by a similar proof to that of Lemma \ref{lem:stab_x}.  
Since $0\in \partial_n\big(F_M(\bar x,\cdot)+i_{\Cs{N}}\big)(n_M)$ and $n_M \to \bar n$, by Lemma~\ref{lem:SAA_phi}, we have $0\in \partial_n\big(F(\bar x,\cdot)+i_{\Cs{N}}\big)(\bar n)$.
%for all $M$, we have $(n_M,0)\in \Graph\,\partial_n\big(F_M(\bar x,\cdot)+i_{\Cs{N}}\big)$. Moreover, uniform convergence implies $F_M(\bar x,n_M)+i_{\Cs{N}}(n_M)\to F(\bar x,\bar n)+i_{\Cs{N}}(\bar n)$.
%Using the closedness property of the limiting subdifferential (cf.\ the closedness of $\Graph\,\partial f$ stated after the definition of $\partial f$), we conclude $(\bar n,0)\in \Graph\,\partial_n\big(F(\bar x,\cdot)+i_{\Cs{N}}\big)$, i.e., $0\in \partial_n\big(F(\bar x,\bar n)+i_{\Cs{N}}(\bar n)\big)$.
Hence $(\bar x,\bar n)\in \Cs{S}$, proving that every sequence in $\Cs{S}_M$ has subsequences converging to points in $\Cs{S}$.
The previous argument implies 
%\begin{equation*}
$\sup_{(x,n)\in\Cs{S}_M}\dist\big((x,n),\Cs{S}\big)\to 0$ as $M\to\infty$. 
%\end{equation*}
Otherwise, 
%Suppose by contradiction that this were false. Then, 
there exist $\varepsilon>0$ and a subsequence $(x_M,n_M)\in\Cs{S}_M$ such that $\dist\big((x_M,n_M),\Cs{S}\big)\ge \varepsilon$, contradicting the previous subsequential limit argument.  

To show Hausdorff convergence, we further need to prove  
%Since $\Cs{X}$ is finite and $\Cs{N}$ is compact, $\{(x_M,n_M)\}$ has a further convergent subsequence with $x_M\equiv \bar x$ and $n_M\to \bar n\in\Cs{N}$. Using the previous argument, the limit point satisfies $(\bar x,\bar n)\in\Cs{S}$, which implies $\dist((x_M,n_M),\Cs{S})\to 0$ along this subsequence, contradicting $\dist((x_M,n_M),\Cs{S})\ge \varepsilon$.
that 
%\begin{equation*}
$\sup_{(x,n)\in\Cs{S}}\dist\big((x,n),\Cs{S}_M\big)\to 0 $ as  $M\to\infty$. 
%\end{equation*}
Fix any $(x^\ast,n^\ast)\in\Cs{S}$. By the quadratic growth at $n^\ast$ for $F(x^\ast,\cdot)$ (see Proof of Theorem \ref{thm:min_convergence}) and uniqueness of the minimizer $x^\ast$ of $F(\cdot,n^\ast)$, there exist $\delta>0$, $c>0$, and $\gamma >0$ such that
$F(x^\ast,n) \ge F(x^\ast,n^\ast) + \frac{c}{2}\|n-n^\ast\|^2$ for all $n\in \Cs{N}\cap \Bs{N}_\delta(n^\ast)$ and
$F(z,n^\ast)\ge F(x^\ast,n^\ast)+\gamma$ for all $z\in\Cs{X}\setminus\{x^\ast\}$. 
Choose $\eta\in(0,\delta)$. Since $F_M\to F$ uniformly on $\Cs{X}\times \Cs{N}$, there exists $M_0$ such that for all $M\ge M_0$ and all $(x,n)\in\Cs{X}\times\Cs{N}$,
$|F_M(x,n)-F(x,n)|\le \frac{c}{8}\eta^2$.    
Then, for any $n\in \Cs{N}\cap \Bs{N}_\delta(n^\ast)$ with $\|n-n^\ast\|\ge \eta$,
\begin{align*}
F_M(x^\ast,n)
&\ge F(x^\ast,n)-\frac{c}{8}\eta ^2
\ge F(x^\ast,n^\ast)+\frac{c}{2}\|n-n^\ast\|^2-\frac{c}{8}\eta^2 
\ge F(x^\ast,n^\ast)+\frac{c}{2}\eta^2-\frac{c}{8}\eta^2\\
&= F(x^\ast,n^\ast)+\frac{3c}{8}\eta^2 \ge F_M(x^\ast,n^\ast)+\frac{3c}{8}\eta^2-\frac{c}{8}\eta^2
=F_M(x^\ast,n^\ast)+\frac{c}{4}\eta^2.
\end{align*}
Hence, $F_M(x^\ast,n)\ge F_M(x^\ast,n^\ast)+\frac{c}{4}\eta^2$ when $\|n-n^\ast\|\ge \eta$ and $M\ge M_0$.
Now consider the restriction of $F_M(x^\ast,\cdot)$ to the compact set $\Cs{N}\cap \Bs{N}_\delta(n^\ast)$. Since $n^\ast\in \Cs{N}\cap \Bs{N}_\delta(n^\ast)$, the minimum
$\tilde n_M\in \argmin_{n\in \Cs{N}\cap \Bs{N}_\delta(n^\ast)} F_M(x^\ast,n) $   
exists. By construction, $\|\tilde n_M-n^\ast\|\le \eta$. Moreover, $\tilde n_M$ is a (constrained) local minimizer of $F_M(x^\ast,\cdot)$ on $\Cs{N}$, and hence, 
$0\in \partial_n\big(F_M(x^\ast,\cdot)+i_{\Cs{N}}\big)(\tilde n_M)$, where this stationarity follows since $\tilde n_M$ minimizes $F_M(x^\ast,\cdot)$ over the neighborhood $\Cs{N}\cap \Bs{N}_\delta(n^\ast)$. %Pick $\tilde x_M = T_M(\tilde{n}_M) \in \argmin_{z\in\Cs{X}} F_M(z,\tilde n_M)$, given that $\argmin_{z\in\Cs{X}} F_M(z,\tilde n_M)$ is nonempty by the finiteness of $\Cs{X}$. 
%Then, $(\tilde x_M,\tilde n_M)\in \Cs{S}_M$ by definition. 
%We claim that every cluster point of any sequence $\{\tilde x_M\}_{M\ge M_0}$ equals $x^\ast= T(n^\ast) \in \argmin_{z\in\Cs{X}}F(z,n^\ast)$. 
%On a 
%Indeed, there exists a 
%further subsequence $M_k$, we have 
%such that $\tilde x_{M_k}=\bar x$ by the finiteness of $\Cs{X}$. By the uniform convergence $F_M\to F$ on $\Cs{X}\times\Cs{N}$ and the fact that 
%$\tilde n_{M_k}\to n^\ast$ as $\eta>0$ is arbitrary. 
By the uniform convergence $F_M\to F$ on $\Cs{X}\times\Cs{N}$ and the fact that $F(z,n^\ast)\ge F(x^\ast,n^\ast)+\gamma$ for all $z\in\Cs{X}\setminus\{x^\ast\}$,  we have 
%\begin{equation*}
%$F_{M_k}(x^\ast,\tilde n_{M_k})\to F(x^\ast,n^\ast)$. 
%\end{equation*}
$x^\ast \in \argmin_{z\in\Cs{X}}F_{M}(z,\tilde n_{M})$. Thus, $(x^\ast,\tilde n_M)\in \Cs{S}_M$, and since $\eta>0$ is arbitrary, $\dist((x^\ast,n^\ast),\Cs{S}_M)\to0$. 
A compactness argument (taking the supremum) over $\Cs{S}$ then yields
%ifor every $z\in\Cs{X}$ we have
%\begin{equation*}
%$F_{M_k}(\tilde x_{M_k},\tilde n_{M_k})\le F_{M_k}(z,\tilde n_{M_k})$.
%\end{equation*}
%Passing to the limit along $k$ yields $F(\bar x,n^\ast)\le F(z,n^\ast)$ for all $z\in\Cs{X}$. Therefore, $\bar x= T(n^\ast) \in \argmin_{z\in\Cs{X}} F(z,n^\ast)$, by a similar proof to that of Lemma \ref{lem:stab_x}. In other words, $\bar{x}=x^\ast$. Therefore, we obtain $\dist((x^\ast,n^\ast),\Cs{S}_M)\to 0$.
$\sup_{(x,n)\in\Cs{S}}\dist((x,n),\Cs{S}_M)\to 0$.
Combining both one-sided limits gives $d_H(\Cs{S}_M,\Cs{S})\to 0$ as $M\to\infty$.
Finally, under the quadratic-growth assumptions of the theorem, stationary points coincide with local minimizers by an argument analogous to that of Theorem \ref{thm:min_convergence}. 
%Therefore, the same Hausdorff convergence conclusion holds. % for the corresponding local-minimizer sets.
\hfill \Halmos
\endproof

\subsection{Convergence Proofs of the PWLA}
% \proof{Proof of Theorem \ref{thm:min_convergence_pwl}}
% By Lemma~\ref{lem:stab_x}, there exists $K_0$ such that $x_M^{k}=x_M^\ast$ for all $k\ge K_0$, where
% \begin{equation*}
% x_M^\ast = T_M(n_M^k) \in \argmin_{x \in \Cs{X}} S_M(x,n_M^{k}),    
% \end{equation*}
% with $P(\cdot;M)=S_M$ in the $x$-step.
% Since $Q(\cdot;M)=F$ in the $y$-step, Lemma~\ref{lem:convergence} yields that $\{n_M^k\}$ converges to some $n_M^\ast\in \Cs{N}$ and satisfies the stationarity condition $0 \in \partial G(n^\ast)$, 
% where $G(n):=F(x^\ast,n)+i_{\Cs{N}}(n)$. That is, $n_M^\ast$ is a critical point of $F(x_M^\ast,\cdot)$ over $\Cs{N}$. The remainder of the argument follows the proof of Theorem \ref{thm:min_convergence} verbatim, yielding that $(x_M^\ast,n_M^\ast)$ is a local minimum of $F$ over $\Cs{X}\times\Cs{N}$.
% \hfill \Halmos
% \endproof

A key to the proof of  Theorem~\ref{thm:pwl_stationary_consistency} is the uniform convergence of $S_M$, as defined in \eqref{eq: pwl}, to $F$, which will be stated next.

\begin{lemma}[Uniform Convergence]
\label{lem:uniform_convergence_PWL}
Suppose the breakpoints used to construct $\Psi_M$ cover an interval
containing
$\Cs{Z}:=\left\{\frac{g_{ij}(x)}{h_j(n_j;\sigma_j)}:
x\in\Cs{X},\ n\in\Cs{N}, \; i\in\Cs{I}, \; j\in\Cs{J}\right\}$,
and let $\kappa_M:=\max_{t\in\{0,\dots,M-1\}}(z^{t+1}-z^t)$, with $\kappa_M\to0$ as $M \to \infty$. 
Then, $\max_{n \in \Cs{N}} \max_{x\in\Cs{X}}|S_M(x,n)-F(x,n)| \to 0$ as $M\to\infty$.   
%\end{equation*}
\end{lemma}
\proof{Proof.}
For simplicity, we suppress $\sigma$ in our notation. 
Since $\Cs{X}$ is finite, $\Cs{N}$ is compact, and $h_j(\cdot)$
is positive, continuous, and bounded on $[0,b]$, $\Cs{Z}$ is bounded.
Moreover, there exists a finite constant $C:=\max_{n \in \Cs{N}} \sum_{j \in \Cs{J}}\sum_{i \in \Cs{I}} q_{ij} h_j(n_j)<\infty $
Since $\Psi'(z)=\Phi(z)\in[0,1]$, $\Psi$ is globally Lipschitz with constant $1$ on $\Bs{R}$. Hence, on each interval $[z^t,z^{t+1}]$, $0 \le t < M$, the linear interpolation error is bounded by the oscillation of $\Psi$ on that interval, so
%\begin{equation*}
$\sup_{z\in \Cs{Z}}|\Psi_M(z)-\Psi(z)|\le \max_{t \in \{0,1\ldots,M-1\}}  \sup_{u,v\in[z^t,z^{t+1}]}|\Psi(u)-\Psi(v)|
\le \max_{t \in \{0,1\ldots,M-1\}} (z^{t+1}-z^t)=\kappa_M$.     
%\end{equation*}
%Therefore, $\sup_{z\in\Bs{R}}|\Psi_M(z)-\Psi(z)|\to 0$ as  $M \to \infty$.
Consequently, for all $x\in\Cs{X}$ and $n \in \Cs{N}$,
%\begin{equation*}
$|S_M(x,n)-F(x,n)|
\le \sum_{j \in \Cs{J}}\sum_{i \in \Cs{I}} h_j(n_j) q_{ij}
\sup_{z\in\Cs{Z}}|\Psi_M(z)-\Psi(z)| \le C \kappa_M$.    
Taking the maximum over $\Cs{X}\times\Cs{N}$ and letting $M\to\infty$ 
prove the claim. 
\hfill \Halmos
\endproof
\proof{Proof of Theorem \ref{thm:pwl_stationary_consistency}.}
%implies that for each fixed $n\in\Cs{N}$, 
%\begin{equation*}
%$\max_{x\in\Cs{X}}|S_M(x,n)-F(x,n)|\to 0$ as $M\to\infty$.
%\end{equation*}
Let $(x_M,n_M)\in \Cs{S}_M$. 
Because $\Cs{X}$ is finite and $\Cs{N}$ is compact, there exists a subsequence (not relabeled) such that
$x_M\equiv \bar x$ and $n_M\to \bar n\in\Cs{N}$. By definition of $\Cs{S}_M$, we have $0\in \partial_n\big(F(\bar x,\cdot)+i_{\Cs{N}}\big)(n_M)$ and $\bar x \in \argmin_{z\in\Cs{X}} S_M(z,n_M)$. By the uniform convergence $S_M \to F$ (Lemma~\ref{lem:uniform_convergence_PWL}) and an argument similar to that in the first part of the proof of Theorem~\ref{thm:saa_stationary_consistency}, we have 
% By the pointwise convergence of $S_M(\cdot,n_M)$ to $F(\cdot,n_M)$ over finite $\Cs{X}$, there exists $M_0(\varepsilon)$ such that for all $M\ge M_0(\varepsilon)$, 
% $|S_M(\bar x,n_M)-F(\bar x,n_M)|\le \varepsilon$ and $|S_M(z,n_M)-F(z,n_M)|\le \varepsilon$.    
% Since $\bar x\in\argmin_{u\in\Cs{X}}S_M(u,n_M)$, we have 
% $F(\bar x,n_M)
% \le S_M(\bar x,n_M)+\varepsilon
% \le S_M(z,n_M)+\varepsilon
% \le F(z,n_M)+2\varepsilon$.    
% Letting $M\to\infty$ and then $\varepsilon\downarrow0$, continuity of $F$ yields $F(\bar x,\bar n)\le F(z,\bar n)$ for all $z\in\Cs{X}$. Therefore, 
$\bar x \in \argmin_{z\in\Cs{X}} F(z,\bar n)$. %, by a similar proof to that of Lemma \ref{lem:stab_x}.
%We now show that 
%\begin{equation*}
%$0\in \partial_n\big(F(\bar x,\cdot)+i_{\Cs{N}}\big)(\bar n)$.
%\end{equation*}
Since $0\in \partial_n\big(F(\bar x,\cdot)+i_{\Cs{N}}\big)(n_M)$, 
%we have $(n_M,0)\in \Graph\,\partial_n\big(F(\bar x,\cdot)+i_{\Cs{N}}\big)$.
by the closedness property of the limiting subdifferential, 
%(cf.\ the closedness of $\Graph\,\partial f$ stated after the definition of $\partial f$), 
we have 
%$(\bar n,0)\in \Graph\,\partial_n\big(F(\bar x,\cdot)+i_{\Cs{N}}\big)$, i.e., 
$0\in \partial_n\big(F(\bar x,\cdot)+i_{\Cs{N}}\big)(\bar n)$.
Hence $(\bar x,\bar n)\in \Cs{S}$, proving that every sequence in $\Cs{S}_M$ has subsequences converging to points in $\Cs{S}$, which also implies $\sup_{(x,n)\in\Cs{S}_M}\dist\big((x,n),\Cs{S}\big)\to 0$ as $M\to\infty$. 

To show Hausdorff convergence, we further need to prove that  
$\sup_{(x,n)\in\Cs{S}}\dist\big((x,n),\Cs{S}_M\big)\to 0 $ as $M\to\infty$. 
Fix any $(x^\ast,n^\ast)\in\Cs{S}$; so,  
$0\in \partial_n\big(F(x^\ast,\cdot)+i_{\Cs{N}}\big)(n^\ast)$.
By the uniform convergence $S_M\to F$ on $\Cs{X}\times\Cs{N}$ and the fact that $F(z,n^\ast)\ge F(x^\ast,n^\ast)+\gamma$ for all $z\in\Cs{X}\setminus\{x^\ast\}$,  we have $x^\ast  \in \argmin_{z\in\Cs{X}}S_{M}(z,n^\ast)$. Thus, $(x^\ast,n^\ast)\in \Cs{S}_M$ for all sufficiently large $M$.  
Thus, $\Cs{S}\subseteq \Cs{S}_M$ for all sufficiently large $M$. 
% Define $\tilde x_M:=T_M(n^\ast)\in\argmin_{z\in\Cs{X}}S_M(z,n^\ast)$. Then, by the definition of $\Cs{S}_M$, we have $(\tilde x_M,n^\ast)\in\Cs{S}_M$ for all $M$.
% It remains to show that $\tilde x_M\to x^\ast$ (along subsequences, which is sufficient since $\Cs{X}$ is finite). Let $\bar x$ be any cluster point of $\{\tilde x_M\}$. Then, there exists a subsequence $M_k$ such that $\tilde x_{M_k}=\bar x$ for all $k$ by the finiteness of $\Cs{X}$. For any $z\in\Cs{X}$, optimality of $\tilde x_{M_k}$ gives
% $S_{M_k}(\bar x,n^\ast)\le S_{M_k}(z,n^\ast)$.    
% Using uniform convergence $S_M\to F$ on $\Cs{X}\times\Cs{N}$ and passing to the limit along $k$, we obtain
% $F(\bar x,n^\ast)\le F(z,n^\ast)$ for all $z\in\Cs{X}$. Therefore, $\bar x= T(n^\ast) \in \argmin_{z\in\Cs{X}} F(z,n^\ast)$, by a similar proof to that of Lemma \ref{lem:stab_x}. In other words, $\bar{x}=x^\ast$. 
% Therefore, every cluster point equals $x^\ast$, so $\tilde x_M\to x^\ast$. Consequently, $\dist\big((x^\ast,n^\ast),\Cs{S}_M\big)\le \|\tilde x_M-x^\ast\|\to 0$.
Taking the supremum over $(x^\ast,n^\ast)\in\Cs{S}$ yields $\sup_{(x,n)\in\Cs{S}}\dist((x,n),\Cs{S}_M)\to 0$.
Combining both one-sided limits gives $d_H(\Cs{S}_M,\Cs{S})\to 0$ as $M\to\infty$.
\hfill \Halmos
\endproof

\end{document}